%% file: arxiv.tex
\documentclass[11pt, a4paper]{article}
\usepackage[margin=1in]{geometry}

\usepackage{tikz} 

\usepackage{algorithm}
\usepackage{algorithmic}
\usepackage{epsfig,amssymb,amsfonts,amsmath,amsthm}
\usepackage{amssymb,amsmath,amsfonts}
\usepackage{multirow}
\usepackage{xspace}
\usepackage{thmtools,thm-restate}

\renewcommand{\Pr}[1]{\mathbb{P}\left[\,#1\,\right]}
\newcommand\E[1]{\mathbb{E}\left[\,#1\,\right]}

\def\tmix{t_{\mbox{\scriptsize{\textsf{mix}}}}}
\def\thit{t_{\mbox{\scriptsize{\textsf{hit}}}}}

\newcommand{\remove}[1]{}

\newcommand{\R}{\mathbb{R}}

\newcommand{\dist}{\operatorname{dist}}
\newcommand{\poly}{\operatorname{poly}}

\newcommand{\eps}{\epsilon}

\newcommand{\argmax}{\operatorname{argmax}}

\renewcommand{\leq}{\leqslant}
\renewcommand{\geq}{\geqslant}
\renewcommand{\le}{\leqslant}
\renewcommand{\ge}{\geqslant}

\newcommand{\Hmax}[2]{\mathsf{H}_{\max}}

\newcommand{\thmref}[1]{Theorem~\ref{thm:#1}}

\newcommand{\lemref}[1]{Lemma~\ref{lem:#1}}

\newcommand{\secref}[1]{Section~\ref{sec:#1}}

\newcommand{\eq}[1]{equation~\eqref{eq:#1}}

\DeclareMathOperator{\Ev}{\mathbb{E}}
\DeclareMathOperator{\var}{Var}

\renewcommand{\epsilon}{\varepsilon}

\usepackage{setspace}

\def\argmax{\operatornamewithlimits{argmax}}

\def\mod{\operatorname{mod}}

\newcommand{\calL}{{\mathcal{L}}}
\newcommand{\calG}{{\mathcal{G}}}
\newcommand{\calE}{{\mathcal{E}}}

\usepackage[linkcolor=red,citecolor=green]{hyperref}

\newtheorem{theorem}{Theorem}[section]  

\newtheorem{lemma}[theorem]{Lemma}
\newtheorem{maintheorem}{Main Result}

\newtheorem{corollary}[theorem]{Corollary}
\newtheorem{remark}[theorem]{Remark}

\newtheorem{proposition}[theorem]{Proposition}

\numberwithin{theorem}{section}
\numberwithin{lemma}{section}
\numberwithin{corollary}{section}
\numberwithin{proposition}{section}
\numberwithin{remark}{section}

\usepackage{mathtools}

\DeclarePairedDelimiter\floor{\lfloor}{\rfloor}

\usepackage{nicefrac}
\usepackage{xfrac}
\usepackage{booktabs} 

\title{Random Walks on Dynamic Graphs: Mixing Times, Hitting Times, and Return Probabilities}
\author{Thomas Sauerwald\thanks{Department of Computer Science and Technology, University of Cambridge, UK} \and Luca Zanetti\footnotemark[1]}
\date{}

\begin{document}

\maketitle

\begin{abstract}
We establish and generalise several bounds for various random walk quantities including the mixing time and the maximum hitting time. Unlike previous analyses, our derivations are based on rather intuitive notions of local expansion properties which allows us to capture the progress the random walk makes through $t$-step probabilities.

We apply our framework to dynamically changing graphs, where the set of vertices is fixed while the set of edges changes in each round. For random walks on dynamic connected graphs for which the stationary distribution does not change over time, we show that their behaviour  is in a certain sense similar to static graphs.
For example, we show that the mixing and hitting times of any sequence of $d$-regular connected graphs is $O(n^2)$, generalising a well-known result for static graphs. We also provide refined bounds depending on the isoperimetric dimension of the graph, matching again known results for static graphs. Finally, we investigate properties of random walks on dynamic graphs that are not always connected: we relate their convergence to stationarity to the spectral properties of an average of transition matrices and provide some examples that demonstrate strong discrepancies between static and dynamic graphs.
\end{abstract}

\input{intro}

\input{notation}

\input{worsthit}

\input{twopluseps}

\input{average}

\input{cutsets}

\bibliography{references2}

\appendix

\input{appendix}

\end{document}

%% file: intro.tex
\section{Introduction}

\textbf{Problem and Motivation.} A random walk is a stochastic process on an undirected connected graph $G=(V,E)$. A particle starts on a specified vertex, and then at each time-step $t=1,2,\ldots$ it moves to a neighbouring vertex chosen uniformly at random. Random walks have proven to be extremely powerful in the design of various sampling schemes, exploration strategies, and distributed algorithms~\cite{LPW06}. They provide a simple yet robust way to explore a large network. Most of the studies on random walks, however, assume the underlying graph to be fixed. In contrast, many prevalent networks today (such as the Internet, social networks, and wireless communication networks) are subject to dramatic changes in their topology over time. Therefore, understanding the theoretical power and limitations of dynamic graphs has been identified as one of the key challenges in computer science~\cite{MS18}.

Recently, several works have considered this problem and investigated the behaviour of random walks \cite{AugustinePR16,avin,DenysyukR14,spirakis,merging1,merging2} or similar processes \cite{icalp2016voter,ClementiCDFPS16,ClementiST15,GiakkoupisSS14,KuhnO11} on such dynamic graphs.
Moreover, rather than a property of the underlying network itself, dynamic graphs may naturally arise in distributed algorithms when communication is performed on a changing, possibly disconnected, subgraph like a spanning-tree or a matching (see, e.g., \cite{gossip}).

One very popular model is that of an evolving graph, where we consider a sequence of graphs $G^{(1)},G^{(2)},\ldots$ over the same set of vertices but with a varying set of edges. This model has been the subject of the majority of previous studies of random walks on dynamic graphs and will be the object of our study as well. Another important feature of dynamic networks is that, with a changing set of edges, the resulting connectivity (i.e., expansion) changes. This might be very common in communication networks, where nodes change their location in space over time and can only communicate if they are within a certain distance of each other. For example,	\cite{KuhnO11} highlights the need to study such evolving graphs with relatively poor connectivity and \cite{MS18} emphasises the unpredictable nature of fast-changing dynamic networks. To incorporate these features into our model, we will consider evolving graphs with relatively mild assumptions on their connectivity and will not make any restriction on how fast they are changing. Our quantitative analysis is focused on the {\em mixing time}, the time to converge to the equilibrium distribution, and the {\em hitting time}, the 
expected number of steps required by a random walk that starts in a vertex $u$ to reach a vertex $v$. Analysing the mixing time of dynamic graphs is also useful for load balancing applications, where the mixing time represents the time it takes for all nodes to have (roughly) a load that is proportional to their stationary distribution. Most theoretical studies of load balancing so far assumed the graph to be fixed.

\textbf{Our Results.} The main motivation for our work comes from the results by Avin et al.~\cite{avin}, which describe a remarkable dichotomy with respect to the behaviour of random walks in evolving graphs: while sequences of connected graphs that share the same stationary distribution are guaranteed to have mixing and hitting times polynomial in the size of the graphs, even small incremental changes to the stationary distribution can cause hitting times to become exponential in the worst case. We focus on the first case of this dichotomy and prove that, at least regarding mixing and hitting times, there is essentially no difference in the behaviour of random walks on static and evolving graphs with a time-independent stationary distribution.

Recall that, for static graphs, it is well-known that the worst-case hitting time is $O(n^2)$ for regular graphs and $O(n^3)$ for arbitrary graphs \cite{Fe95a,Fe95b}. Quite surprisingly, we can show that something very similar holds in the setting of evolving graphs: our theorem below 
 proves an upper bound of $O(n^2)$ for the mixing and hitting times of regular evolving graphs, which is optimal even for static graphs, an upper bound of $O(n^3)$ for the mixing time of non-regular evolving graphs, which is again optimal even for static graphs, and an $O(n^3 \log{n})$ upper bound for the maximum hitting time, which is only a factor of $O(\log{n})$ short from the optimal bound on static graphs (simply consider the Barbell graph, i.e., two cliques of size $n/3$ connected by a path of length $n/3$, which has $O(n^3)$ mixing and maximum hitting time).

\begin{maintheorem}[{restated, see Theorem~\ref{thm:worsthit} on page~\pageref{thm:worsthit}}]\label{thm:firstmain}
Let $\mathcal{G}=\{G^{(t)}\}_{t=1}^{\infty}$ be a sequence of connected graphs with $n$ vertices, the same stationary distribution $\pi$, and at most $m^*$ edges. Then:
\begin{enumerate}\itemsep0pt
	\item $\tmix(\mathcal{G}) = O(n/\pi_*)$,
	\item $\bigl| \frac{p^{[0,t]}_{u,v}}{\pi_v} - 1 \bigr| \lesssim \frac{m^*}{t} + \frac{1}{\pi_* \sqrt{t}}$, simplifying to $\bigl| \frac{p^{[0,t]}_{u,v}}{\pi_v} - 1 \bigr| \lesssim \frac{n}{\sqrt{t}}$  if all graphs in $\mathcal{G}$ are $d$-regular,
	\item $\thit(\mathcal{G}) = O(n \log{n}/\pi_*)$. Furthermore, if the graphs in $\mathcal{G}$ are $d$-regular,  $\thit(\mathcal{G}) = O(n^2)$.
\end{enumerate} 
\end{maintheorem}

\begin{remark}
In this work, we never explicitly derive upper bounds on the cover time (i.e., the expected time for a random walk to visit all vertices). However, analogous to Matthew's Bound for static graphs~\cite[Chapter~11.2]{LPW06}, all the {\em stated} upper bounds on hitting times can be converted into upper bounds on cover times at the cost of an additional $O(\log n)$-factor.
\end{remark}
\begin{remark}
Unlike static graphs where the gap between cover and hitting time is $O(\log n)$ (thanks to~Matthew's Bound), for evolving graphs the gap could be $\Omega(n)$ even if the sequence consists of regular connected graphs. For example, for any $t \leq c n \ln n$, let $G^{(t)}$ be a complete graph with $n$ vertices, while for any $t > c n \ln n$, let $G^{(t)}$ be a cycle with $n$ vertices. We can choose the constant $c$ so that, with probability $1-\Theta(n^{-1})$, every fixed vertex is visited before $c n \ln n$ steps, but with constant nonzero probability, there is at least one unvisited vertex which is at distance $\Omega(n)$ from the location of the walk at step $c n \ln n$. This yields a $\Theta(n)$ maximum hitting time, but a $\Theta(n^2)$ cover time.
\end{remark}

A natural question is of course under which conditions the worst-case bound on the hitting time can be improved. For static graphs, it has been observed that for many regular networks, the hitting time is indeed optimal, i.e., $O(n)$. One very general and unifying condition 
is the conjecture of Aldous and Fill~\cite[Open Problem 6.20]{aldousFill} stating that for any bounded-degree, $d$-regular graph, an isoperimetric dimension of $2+\epsilon$ is enough for hitting times to be linear (which is as good as possible). Since the isoperimetric dimension is equal to the dimension of grids, it follows that grids of dimension $3$ or higher have a linear hitting time, while grids of dimension $2$ have a hitting time of $O(n \log n)$. 

For static graphs, a positive answer to the above conjecture by Aldous and Fill was first given by \cite{benjamini}, and another proof was found by \cite{gharan}. Both these proofs, however, exploit the connection between hitting times and electrical resistances~\cite{chandra}, which is not known to exist for evolving graphs (indeed, it is not even clear how to formalise such connection). Since our techniques for bounding hitting times are more probabilistic in nature and avoid arguments based on electrical networks, we are able to show that the conjecture by Aldous and Fill is true even in a dynamic setting. 
\begin{maintheorem}[restated, see Theorem~\ref{thm:twopluseps} on page~\pageref{thm:twopluseps}]
\label{main:twopluseps}
Let $\mathcal{G}=\{G^{(t)}\}_{t=1}^{\infty}$ be a sequence of $n$-vertex graphs such that each $G^{(t)}$ is regular, has bounded degree, and satisfies the following isoperimetric condition: there exists $\eps \in [0,1/4]$ such that, for any subset of vertices $A$ with $1 \le |A| \le n/2$, $|E(A,V \setminus A)| = \Omega(|A|^{\frac{1}{2} + \eps})$. Then,
\begin{enumerate}\itemsep0pt
\item $\tmix(\mathcal{G}) = O(n^{1-2\eps})$,
\item $\bigl| \frac{p^{[0,t]}_{u,v}}{n} - 1 \bigr| = O\left(\frac{1}{t^{1+2\eps}}\right)$,
\item $\thit(\mathcal{G}) = O(n)$ if $\eps>0$, $\thit(\mathcal{G}) = O(n\log{n})$ if $\eps=0$.
\end{enumerate} 
\end{maintheorem}

Note that the isoperimetric condition essentially says that each graph in the sequence must be at least as well-connected as a $2+\epsilon$-dimensional grid. For $\epsilon=0$, we recover the $O(n \log n)$ hitting time for static two-dimensional grids. Both of these cases might be relevant in certain applications of moving wireless devices or robots performing terrain exploration.

The first two results apply to settings where there is a ``stable connectivity'', but each graph in the sequence may have a relatively poor expansion. The next result applies to scenarios where connectivity is more intermittent, in fact some of the vertices may even be isolated at some time steps. However, ``averaging'' over a sufficiently long time window, the graph will not only  be connected but might also satisfy some reasonably strong expansion guarantee. In this sense, this model is somewhat related to that of \cite{KLO10}, which stipulates the existence of a spanning subgraph over any time-interval of a certain length. 
More formally, in the next theorem we assume that the random walk is on a sequence $\calG$ of graphs with transition matrices $P^{(1)},P^{(2)},...$ and there exists a time-independent distribution $\pi$ which is stationary for any $P^{(i)}$. We remark that we do not assume connectivity and, therefore, any individual $P^{(i)}$ might have multiple stationary distributions. We assume, however, that there exists a large enough time window $t$ such that, for any $i \ge 1$, $\overline{P}^{[i,i+t]} =\frac{1}{t}(P^{(i)} + P^{(i+1)} + \cdots + P^{(i+t)})$ is ergodic with a unique stationary distribution $\pi$ and spectral gap $\lambda(\overline{P}^{[i,i+t]}) \ge \lambda > 0$. We then can show that the distribution of  a lazy random walk on $\calG$ converges to $\pi$ at a rate that depends on $t$ and the spectral gap $\lambda$. We refer to \secref{average} for details on the set-up.
\begin{maintheorem}[restated, see Corollary~\ref{cor:average} on page~\pageref{cor:average}]
\label{main:average}
Consider a dynamically evolving sequence $\mathcal{G}=\{G^{(t)}\}_{t=1}^{\infty}$ of graphs with transition matrices $\{P^{(i)}\}_{i=1}^{\infty}$ such that (1) there exists $\pi$ which is a stationary distribution for any $P^{(i)}$; and (2) there exists a time-window $t \ge 0$ such that, for any $i \ge 0$, $\overline{P}^{[i,i+t]} =\frac{1}{t}(P^{(i)} + P^{(i+1)} + \cdots + P^{(i+t)})$ is ergodic and has spectral gap $\lambda\bigl(\overline{P}^{[i,i+t]}\bigr) \ge \lambda > 0$. Then,
$\tmix(\calG) = O\left(t^2\log(1/\pi_*)\lambda^{-1}\right)$, where $\pi_* = \min_u \pi(u)$.
\end{maintheorem}

This result is not only significant in the context of dynamically evolving graphs, but also in settings of static graphs where communication is restricted to a bounded-degree subgraph which potentially changes in each round. One prominent example are matching-based communications, where in each round a random matching is generated and only those edges can be used for averaging or exchanging information, e.g.,~\cite{gossip}.

Even when the assumptions of Main Theorem \ref{main:average} are satisfied for a \emph{small} time-window $t$, we cannot always guarantee that hitting and mixing times will be polynomial in the size of the graphs. Indeed, we exhibit examples of dynamic evolving graphs of $n$ vertices that satisfy such conditions but have mixing and/or hitting times that are exponential in $n$ and $t$. We show in Proposition \ref{pro:nohitting} that, since graphs in the sequences need not be connected, it is possible to construct examples where the stationary distribution $\pi$ has exponentially small probability mass on some vertices. This could result in exponential mixing and hitting times, but somewhat surprisingly also possibly in polynomial mixing time and exponential maximum hitting time. Both of the constructed graph sequences rely on the idea of simulating a directed graph by a sequence of disconnected bipartite graphs.

A natural question is whether we can relax the assumptions on the regularity or existence of a time-invariant stationary distribution. Unfortunately, we show that this is not always possible. We exhibit in Proposition \ref{pro:nomixing} a sequence of graphs which are connected, have bounded-degree and constant spectral gap, but for which $t$-step probabilities are very far from the uniform distribution even for a time $t$ which is larger than the mixing time of a random walk on any (static) graph in the sequence.

Going back to Main Result~\ref{main:twopluseps}, the essence behind the proof is that, to achieve an optimal $O(n)$ hitting time, we do not need large sets to have high expansion. What we only need is that small sets have a ``sufficiently high'' expansion. We derive another result in the same spirit by upper bounding a variational characterisation of the commute time in terms of some version of the conductance profile~\cite{conductanceprof}. However, since we need to exploit a variational characterisation of the commute time, as opposed to the earlier results, this bound only holds for static graphs. Let $C_{st}$ be the commute time from $s$ to $t$. Then, we have the following.

\begin{maintheorem}[restated, see \lemref{general} on page~\pageref{lem:general}]
\label{main:cutsets}
	For any static graph $G = (V,E)$ and $s,t \in V$, there exists a labelling of the vertices from $1$ to $n$ such that
	$
	C_{st} \le 2m \sum_{j=1}^{n-1}|\partial [j]|^{-1}$,
	where $\partial [j]$ is the set of edges with one endpoint in $\{1,\dots,i\}$ and one in $\{i+1,\dots,n\}$.
\end{maintheorem}

Note the relation between Main Theorem \ref{main:cutsets} and the well-known Nash-Williams' inequality \cite[Proposition 9.15]{levinPeres} which states that, for every set $\{E_1,E_2,\ldots,E_k\}$ of edge-disjoint cut-sets separating $s$ from $t$, 
$C_{st} \geq 2m \sum_{j=1}^{k} |E_j|^{-1}$.
Our upper bound, however, differs from the Nash-Williams' inequality in two ways: (1) the cut-sets $\partial[j]$ are in general not edge-disjoint; (2) we prove the existence of a ``good'' labelling, while Nash-Williams holds for \emph{any} labelling. 

 
As an application of this result, we consider the hitting time on $d$-regular graphs in terms of the edge-connectivity, which does not impose any condition on the expansion of large sets.
 
 \begin{maintheorem}[restated, see Theorem~\ref{thm:connectivity} on page~\pageref{thm:connectivity}]
Let $G=(V,E)$ be any static $d$-regular graph with edge-connectivity $\rho$. Then $\thit(\mathcal{G}) \leq O(n^2 \cdot ( \frac{\log d}{d} + \frac{1}{\rho} ) )$. In particular, since $\rho \leq d$, we get the simpler (but potentially slightly weaker) upper bound $\thit(\mathcal{G}) = O(n^2 \log d/\rho)$.
\end{maintheorem}

We remark that in Aldous and Fill~\cite[Proposition 6.22]{aldousFill}, it was shown that for any $d$-regular graph $G$ which is $\rho$-edge-connected, the maximum hitting time is $O(n^2 d \cdot \rho^{-3/2})$. They also mention that if the graph is $\Omega(d)$-edge-connected, they obtain a bound of $O(n^2 \cdot d^{-1/2})$. For this case of maximal edge-connectivity, $\rho=\Theta(d)$, 
our bound is considerably better than the one by Aldous and Fill, and, modulo the $\log d$-factor, gives also the right dependency on $d$. In particular,
we demonstrate in \secref{cutsets} that the dependency on the edge-connectivity $\rho$ is as good as possible (neglecting logarithmic factors) in the sense that for any pair of $\rho$ and $d$, there exists a $d$-regular graph with edge-connectivity $\rho$ which matches the upper bound in Main Result 5 (\thmref{connectivity}) up to constant factors.

\textbf{Further Related Work.} While in this work we focus on standard (lazy) random walks on graphs, we should point out that previous work has established an alternative in form of the so-called {\em max-degree walk} \cite{avin}. In this random walk variant, a large loop probability depending on the degree of the current vertex and (an estimate of) the maximum degree $\Delta$ is added. With this modification, the stationary distribution of each graph is identical (and uniform), which makes the analysis of this walk easier. However, one downside of this approach is that it either requires a good estimate of $\Delta$ (or even $n$), or the random walk may potentially be slowed down significantly. 
Also, studying standard random walks seems more natural and, as we will see later, it also helps us to uncover some of the subtle boundaries between fast mixing and polynomial hitting, and slow mixing and exponential hitting.

%

%
%

One of the earliest appearances of dynamic graphs is in the context of load balancing~\cite{GhoshLMMPRRTZ99}, where the authors assumed a uniform (i.e., time-independent) lower bound on the edge and vertex expansion. A refinement is to instead relate the balancing (mixing) time to the geometric mean of the spectral gaps, which was used in~\cite{ElsasserMS04}. A result of a similar flavour for both the conductance and the vertex expansion was shown in \cite{GiakkoupisSS14} in the context of randomised rumour spreading, and more recently a similar result was shown for the cover model \cite{icalp2016voter}. In \cite{KLO10}, the authors analyse a sequence of graphs satisfying a $T$-interval connectivity property, which asserts that for every $T$ consecutive rounds there exists a stable connected spanning subgraph. The authors present upper bounds for several distributed computational problems.

One specific graph model that has been very popular is the so-called Markovian evolving graph. In this model every edge is associated to the same but independent two-state birth and death chain which decides whether the edge is present or not in the next step. Many aspects of this network have been studied, most notably the (dynamic) diameter \cite{ClementiMMPS10} and the time to spread a piece of information \cite{ClementiST15}. Recently, however, Lamprou et al.~\cite{spirakis} also considered the cover time of these graphs. In particular, suppose there exists an underlying graph $G$ with minimum degree $\delta$ such that at each time $t$ the graph $G^{(t)}$ contains each edge in $G$ independently with probability $p$ (i.e., the presence of an edge does not depend on the past). They show the cover time of such dynamic graph is at most $t_{\sf cov}(G) / (1-(1-p)^{\delta})$, where $t_{\sf cov}(G)$ is the cover time of $G$. They also study \emph{random walks with a delay}, where at each step a particle chooses a random neighbour of the current vertex according to the topology of the underlying graph $G$, and moves there if the corresponding edge is present, otherwise waits till it becomes available. For this perhaps slightly less natural process, they give bounds on the cover time also for the case where the probability of an edge being available at time $t$ depends on whether that edge was available at time $t-1$. We also highlight dynamical percolation, a particular type of Markovian evolving graphs that has received recent attention (see, e.g., \cite{hermonSousi,pss15,sousiThomas}). Here, an ``open'' edge becomes ``close'' with probability $p$, while a close edge becomes open with probability $1-p$. In contrast to the literature above, however, works on random walks on dynamical percolation usually refer to continuous-time random walks.


Another class of dynamic graph models involves agents that move in some bounded space and can interact only if they are close enough~\cite{LamLMSW12,PeresSSS11,PettarinPPU11}. In contrast to these works, our bounds are less tight but hold under much weaker assumptions on the graph and therefore capture a more dynamic and less ``regular'' setting.

Finally, we mention that Saloff-Coste and Zuniga~\cite{merging1,merging2} have generalised spectral and geometric techniques, such as Nash and log-Sobolev inequalities, to time-inhomogeneous Markov Chains (of which random walks on dynamic graphs are a subset). In particular, in contrast to our results, they study chains where the individual transition matrices might not have the same time-independent stationary distribution. For this reason, they focus on \emph{merging} properties of these chains, i.e., the ability of the chain to ``forget'' the initial distribution. They obtain bounds on merging for chains that satisfy the $c$-stability property, which implies (but it is not equivalent) that the stationary distributions of the individual transition matrices do not change too much over time. Unfortunately, proving that a time-inhomogeneous chain is $c$-stable is itself very difficult, and they are able to obtain concrete bounds on merging only for very simple time-inhomogeneous Markov chains.



%% file: notation.tex
\section{Notation and preliminaries}\label{sec:notation}
Let  $\mathcal{G} = \{G^{(t)}\}_{t=1}^{\infty}$ be an infinite sequence of undirected and unweighted graphs defined on the same vertex set $V$, with $|V| = n$.
We study (lazy) random walks on $\mathcal{G}$: suppose that at a time $t\ge 0$ a particle occupies a vertex $u \in V$. At step $t+1$ the particle will remain at the same vertex $u$ with probability $1/2$, or will move to a random neighbour of $u$ in $G^{(t)}$. In other words, it will perform a single random walk step according to a transition matrix $P^{(t)}$, which is the transition matrix of a lazy random walk on $G^{(t)}$: $P^{(t)}(u,u) =1/2 $, $P^{(t)}(u,v) = 1/(2d_u)$ if there  is an edge between $u$ and $v$ in $G^{(t)}$ (and in this case we write $u \sim_t v$), or $P^{(t)}(u,v) = 0$ otherwise.
We denote with $p^{[t_1,t_2]}_{u,v}$ the probability that a random walk that visits vertex $u \in V$ at time $t_1$ will visit $v \in V$ at time $t_2 \ge t_1$. Notice that given an initial probability distribution $p^{(0)} : V \to [0,1]$, $p^{[0,t]} = p^{(0)} P^{[0,t]} = p^{(0)} P^{(1)} P^{(2)} \cdots P^{(t)}$ is the probability distribution after $t$ steps.

Unless stated otherwise we assume that all the graphs in $\mathcal{G}$ are connected and have the same stationary distribution $\pi$, i.e., $\pi P^{(t)} = \pi$ for any $t \ge 0$. We denote the smallest value assumed by $\pi$ as $\pi_* = \min_{x \in V} \pi(x)$. We define the $\ell_2(\pi)$-inner product as
$ \langle f,g \rangle_{\pi} = \sum_{u \in V} f(u) g(u) \pi(u)$ for any $f,g: V \to \R$. Analogously, we denote with $\| f \|_{2,\pi} = \sqrt{\langle f,f \rangle_{\pi}}$ the $\ell_2(\pi)$-norm of $f: V \to \R$. Notice that since all the graphs in $\mathcal{G}$ are undirected, for any $t\ge 0$, $P^{(t)}$ is reversible with respect to $\pi$, i.e., $\pi(x) P^{(t)}(x,y) = \pi(y) P^{(t)}(y,x)$ for any $x,y \in V$ (this is also called the detailed balance condition). Moreover, $P^{(t)}$ is self-adjoint for the $\ell_2(\pi)$-inner product: for any $f,g: V \to \R$,
\begin{equation}
\label{eq:selfadjoint}
\langle P^{(t)} f, g \rangle_{\pi} = \langle f, P^{(t)} g \rangle_{\pi}.
\end{equation}

We will often work with the likelihood ratio $\rho^{[0,t]}_{u,\cdot} = p^{[0,t]}_{u,\cdot} / \pi(\cdot)$. When it is clear from the context, we will drop the starting point $u$ and use the shorthands $p^{(t)}$ and $\rho^{(t)}$ to indicate (respectively) the probability distribution of the random walk at time $t$ and its likelihood ratio. We define the $\ell_2$ mixing time as 
\[
\tmix(\mathcal{G}) = \min \{t \colon \| \rho^{[0,t]}_{u,\cdot} - 1 \|_{2,\pi} \le 1/3 \text{ for any } u \in V\}.
\]
Observe that, since  $\Ev_{\pi} \rho^{(t)} = 1$, we have that $\| \rho^{(t)} - 1 \|_{2,\pi}^2 = \var_{\pi} \rho^{(t)} = \Ev_{\pi} \left(\rho^{(t)}\right)^2 - 1$.

Let $p$ be a probability distribution with likelihood ratio $\rho = p/\pi$. For a reversible $P$, 
\[
P\rho (u) = \sum_{v \in V} P(u,v) \rho(v) = \sum_{v \in V} P(u,v) \frac{p(v)}{\pi(v)} = \frac{1}{\pi(u)} \sum_{v \in V} P(v,u) p(v) = \frac{pP (u)}{\pi(u)},
\]
from which it follows that
$P^{(t)} \cdots P^{(1)}\rho^{(0)}(u) =   \rho^{(t)}(u)$.

Given a transition matrix $P$ with stationary distribution $\pi$ and a function $f \colon V \to \R$, we define the Dirichlet form as
\[
\mathcal{E}_P(f,f) = \frac{1}{2} \sum_{u,v \in V} \left(f(u) - f(v)\right)^2 \pi(u) P(u,v).
\]
When $P$ is a transition matrix of a lazy random walk on a graph $G = (V,E)$ with $|E| = m$,
$\mathcal{E}_P(f,f) = \frac{1}{4m} \sum_{u \sim v} \left(f(u) - f(v)\right)^2$,
where $u \sim v$ stands for $\{u,v\} \in E$. As long as $P$ is lazy (i.e., $P(u,u) \ge 1/2$ for any $u \in V$), we can relate the $\ell_2^2$ distance of a distribution from stationary to its Dirichlet form \cite[Proposition 2.5]{nonrevFill}:
\begin{equation}
\label{eq:mihai}
\var_{\pi} \rho^{(t)}  \ge \var_{\pi} \rho^{(t+1)} + \mathcal{E}_{P^{(t+1)}}(\rho^{(t)} ,\rho^{(t)}).
\end{equation}
The \emph{spectral gap} of $P$ is defined as
\[
\lambda(P) = \inf _{\substack{f: V \to \R \\ \var_{\pi}{f} \neq 0}} \frac{\mathcal{E}_P(f,f)}{\var{f}}.
\]
We denote with $\Phi_P(A)$ the \emph{conductance} of a subset of vertices $A \subset V$ :
\[
\Phi_P(A) = \frac{\sum\limits_{{u \in A , v \not\in A}} \pi(u) P(u,v)}{\min\left\{\pi(A), \pi(V \setminus A)\right\}},
\]
where $\pi(A) = \sum_{u \in A} \pi(u)$. The conductance of $P$ is then defined as
$\Phi(P) = \min_{A \subset V} \Phi_P(A)$.
Cheeger's inequality \cite{SJ89} relates $\lambda(P)$ to the conductance $\Phi(P)$ of a reversible $P$: $2\Phi(P) \ge \lambda(P) \ge \Phi(P)^2/2$.

Given two vertices $u$ and $v$, we denote with $\tau_{u,v}$ the \emph{hitting time} of $v$ from $u$, i.e., the expected time to reach $v$ starting from $u$. The maximum hitting time is defined as $\thit(\mathcal{G}) = \max_{u,v} \tau_{u,v}$.   The \emph{commute time} between $u$ and $v$ is instead defined as $C_{u,v} = \tau_{u,v} + \tau_{v,u}$. Clearly, $\thit(\mathcal{G}) \le \max_{u,v} C_{u,v} \le 2\thit(\mathcal{G})$.  

Finally, we write $A \lesssim B$, respectively $A \gtrsim B$, to mean that there exists some absolute constant $C > 0$, independent of the parameters of the sequence of graphs $\mathcal{G}$, such that  $A \le C \cdot B$, respectively $A \ge C \cdot B$.

%% file: worsthit.tex
\section{Worst-case bounds for mixing and hitting times}
\label{sec:worsthit}
In this section we assume that a particle performs a random walk on a sequence of graphs $\mathcal{G} = \{G^{(t)}\}_{t=1}^{\infty}$ where all the $G^{(t)}$ share the same set of $n$ vertices $V$, are connected, and have a time-independent stationary distribution $\pi$ with  $\pi_* = \min_u \pi(u)$. In general, graphs in the sequence might have a different number of edges. We denote with $m^* \le n^2$ the maximum number of edges a graph in the sequence can have. 

Our goal is to bound mixing and maximum hitting times of a random walk on $\mathcal{G}$. We start by studying the rate of convergence to stationarity. By \eq{mihai}, our goal then becomes to study $\var \rho^{(t)} \le \var{\rho^{(0)}} - \sum_{i=1}^{t-1} \mathcal{E}_{P^{(i)}}(\rho^{(i-1)} ,\rho^{(i-1)} )$. The next lemma provides a lower bound on the Dirichlet form of graphs in $\mathcal{G}$. The main insight of this lemma is that it shows a faster decrease of the $\ell_2$-distance to stationarity when this distance is large, i.e., at the beginning of the walk. This is in the same vein as, for example, bounds on mixing based on the spectral profile~\cite{spectralprof}.

\begin{lemma}
\label{lem:dirichlet}
 Let $P$ be the transition matrix of a lazy random walk on a graph $G \in \mathcal{G}$. Given a probability distribution $\sigma: V \to [0,1]$ with likelihood ratio $f = \sigma/\pi$ such that $\var_{\pi} f = \eps  > 0$, 
\[
\mathcal{E}_{P}(f,f) \gtrsim \max\left\{ \frac{\eps^2}{m^* + 1/(\pi_*^2(1+\eps))},  \frac{\pi_* \eps^2}{ n} \right\}.
\]
 \end{lemma}

While the previous lemma will be directly used to derive bounds on mixing, to obtain a bound on the hitting time we will need to study $t$-step probabilities. For this reason,
 we prove a technical lemma that relates $2t$-step probabilities to the variance of the likelihood ratio of a $t$-step probability distribution, generalising a well-known result for time-homogeneous reversible Markov chains (see, e.g., \cite[Lemma 3.20]{aldousFill}). We remark, however, that while in time-homogeneous Markov chains $2t$-step transition probabilities will be as small as the variance of their $t$-step likelihood ratio, in our case, since the order in which transition matrices are applied can matter significantly, this might not be necessarily true: we can only relate these probabilities to the variance of the $t$-step likelihood ratio of a related but slightly different Markov chain.

\begin{lemma}
\label{lem:inftoell2}
Let $t_1 < t_2$. Then, for any $u,v \in V$, it holds that
\[
\left|\rho^{[t_1,t_2]}_{v,u} - 1 \right| \le \max\left\{\var_{\pi} \left( P^{(\floor{\frac{t_1+t_2}{2}})} \cdots P^{(t_2)}  \rho^{[t_1,t_1]}_{u,\cdot}\right) ,
				 \var_{\pi} \left(P^{(\floor{\frac{t_1+t_2}{2}-1})} \cdots P^{(1)}\rho^{[t_1,t_1]} _{v,\cdot} \right) \right\}.
\]
\end{lemma}
 
Using \lemref{dirichlet} and \lemref{inftoell2} we can obtain almost optimal worst case bounds on mixing, hitting, and $t$-step probabilities of a random walk on $\mathcal{G}$. In particular, when $\mathcal{G}$ comprises only regular graphs, the next theorem implies a $O(n^2)$ bound on mixing and hitting times, which matches the well-known results for a random walk on a static undirected graph. In the general non-regular case, we prove a $O(n^3)$ bound on mixing and a $O(n^3 \log{n})$ bound on hitting, which almost matches the $O(n^3)$ bound for mixing and hitting in static graphs. This improves upon \cite{avin}, which presents a bound of $O(n^3 \log{n})$ for hitting on regular graphs and a bound of $O(n^5 \log{n})$ for hitting in the general case.

\begin{theorem}
\label{thm:worsthit}
Let $\mathcal{G}$ be a sequence of connected graphs with $n$ vertices, the same stationary distribution $\pi$, and at most $m^*$ edges in each graph. Then, for a lazy random walk on $\mathcal{G}$:
\begin{enumerate}\itemsep0pt
\item $\tmix(\mathcal{G}) = O(n/\pi_*)$,
\item $\bigl| \frac{p^{[0,t]}_{u,v}}{\pi_v} - 1 \bigr| \lesssim \frac{m^*}{t} + \frac{1}{\pi_* \sqrt{t}}$, simplifying to $\bigl| \frac{p^{[0,t]}_{u,v}}{\pi_v} - 1 \bigr| \lesssim \frac{n}{\sqrt{t}}$  if all the graphs in $\mathcal{G}$ are $d$-regular,
\item $\thit(\mathcal{G}) = O(n \log{n}/\pi_*)$. Furthermore, if the graphs in $\mathcal{G}$ are $d$-regular,  $\thit(\mathcal{G}) = O(n^2)$.
\end{enumerate} 
\end{theorem}

The proof, which is deferred to the appendix, proceeds as follows. First we establish the bound on the mixing time based on \lemref{dirichlet}, which readily implies that starting from a distance to stationarity equal to $\eps$, such distance is halved in $O(n/(\eps \pi_*))$ steps. We then connect distance to stationarity to $t$-step probabilities with \lemref{inftoell2}, obtaining the second result of \thmref{worsthit}. Finally, to bound the hitting time, we employ a probabilistic argument already exploited in, e.g., \cite{KMS16}, and which makes use of both our bounds on mixing time and on $t$-step probabilities. 

%% file: twopluseps.tex
\section{Bounds on hitting times based on the isoperimetric dimension}
\label{sec:twopluseps}
Aldous and Fill conjectured in their book \cite[Open Problem 6.20]{aldousFill} that whenever a regular bounded-degree graph satisfies $|E(A,A^c)| = \Omega(|A|^{\frac{1}{2}+ \eps})$ for any small positive $\eps$,  the maximum hitting time should be $O(n)$. 
Observe that this isoperimetric condition is satisfied by the torus in $3$ or higher dimensions, which has indeed $O(n)$ maximum hitting time. Furthermore, to have $O(n)$ maximum hitting time, $\eps$ needs to be strictly greater than zero: take for example the $2$-dimensional torus: there is a set $A$ for which $|E(A,A^c)| = \Theta(|A|^{1/2})$ and, indeed, the maximum hitting time is $\Theta(n \log{n})$. 

The conjecture was first proved in \cite{benjamini}, with a proof based on the relation between commute times and effective resistances in a graph. Since a similar relation is not known for time-inhomogeneous Markov chains, such a proof cannot be generalise to random walks on dynamic graphs. In this section we present a new proof of this result based on the ``conditional expectation trick'' already used in the proof of \thmref{worsthit}. We start by obtaining a bound on the Dirichlet form of a graph satisfying the aforementioned isoperimetric condition. 

\begin{lemma}
\label{lem:decrease_eps}
Let $G=(V,E)$ be a $d$-regular undirected graph with $|V| = n$ and $d=O(1)$ such that, for any $A \subset V$ with $1 \le |A| \le n/2$,
$
|E(A,V \setminus A)| = \Omega(|A|^{\frac{1}{2} + \eps})
$ for $1/4 \ge \eps \ge 0$. Consider the transition matrix $P$ of a lazy random walk in $G$. Let $\sigma$ be any probability distribution and $f = \sigma/\pi$, where $\pi$ is the uniform distribution.  If $\Ev_{\pi} f^2 = \beta > C$ for a large enough constant $C$, then
\[
\mathcal{E}_P(f,f) \gtrsim  \frac{\beta^{2-2\eps}}{n^{1-2\eps}}.
\]
\end{lemma}

We now apply the previous lemma to prove the main result of this section, in an analogous way to the proof of \thmref{worsthit}.

\begin{theorem}
\label{thm:twopluseps}
Let $\mathcal{G}=\{G^{(t)}\}_{t=1}^{\infty}$ be a sequence of $n$-vertex graphs such that each $G^{(t)}$ is regular, has bounded degree, and satisfies the following isoperimetric condition: there exists $\eps \in [0,1/4]$ such that, for any subset of vertices $A$ with $1 \le |A| \le n/2$, $|E(A,V \setminus A)| = \Omega(|A|^{\frac{1}{2} + \eps})$. Then,
\begin{enumerate}\itemsep0pt
\item $\tmix(\mathcal{G}) = O(n^{1-2\eps})$,
\item $\bigl| \frac{p^{[0,t]}_{u,v}}{n} - 1 \bigr| = O\left(\frac{1}{t^{1+2\eps}}\right)$,
\item $\thit(\mathcal{G}) = O(n)$ if $\eps>0$, $\thit(\mathcal{G}) = O(n\log{n})$ if $\eps=0$.
\end{enumerate} 
\end{theorem}

%% file: average.tex
\section{Bounds on mixing based on average transition probabilities}
\label{sec:average}
Unlike in the time-homogeneous case, eigenvalues of the individual transition matrices of a time-inhomogenous Markov chain are not necessarily indicative of its mixing time, even when there exists a unique time-independent stationary distribution. 
An emblematic example is the following: consider a sequence of graphs $\mathcal{G} = \{G^{(t)}\}_{t=1}^{\infty}$ defined over a vertex set $V = \{1,\dots,2n\}$ such that, at odd $t$, $G^{(t)}$ is the union of two expanders (graphs with constant spectral gap), one over $\{1,\dots,n\}$, the other over $\{n+1,\dots,2n\}$, while at even $t$, $G^{(t)}$ is a perfect matching between $\{1,\dots,n\}$ and $\{n+1,\dots,2n\}$. Since all the graphs are disconnected, each transition matrix has spectral gap equals to $0$, and eigenvalue bounds are, in this case, useless to analyse convergence to stationarity. On the other hand, it is quite clear that a lazy random walk on $\mathcal{G}$ mixes in $\Theta(\log{n})$ time.

A more precise way to study mixing in time-inhomogeneous random walks would be to consider the spectral gap of the product of the transition matrices $P^{(1)} \cdots P^{(t)}$. Unfortunately, spectral bounds for the product of matrices are notoriously hard to come by. What is significantly easier is to study the \emph{average} transition matrix $\overline{P} =\frac{1}{t}\left(P^{(1)} + P^{(2)} + \cdots + P^{(t)}\right)$, which at least does not dependent on the order in which the transition matrices appear. For this reason, in this section we give bounds on mixing on $\mathcal{G}$ that depend on the Dirichlet form of $\overline{P}$. In particular, consider the aforementioned example where $G^{(t)}$ is two disjoint expanders at odd times, and a perfect matching between the two sets at even times. Consider the average transition matrix $\overline{P} =\frac{1}{2}\left(P^{(\ell)} + P^{(\ell + 1)}\right)$ for any two consecutive steps $\ell,\ell + 1$: $\overline{P}$ is just the transition matrix of a random walk on an expander graph defined over the entire set of vertices. Our results, then, make us easily derive the correct bound $\tmix(\mathcal{G}) = O(\log{n})$.

Throughout this section we assume that $\mathcal{G} = \{G^{(t)}\}_{t=1}^{\infty}$ is a sequence of undirected graphs over a vertex set $V$ with $|V| = n$. The graphs are not necessarily connected, which means they might have multiple stationary distributions. We require, however, that there exists a time-independent distribution $\pi$ which is a stationary distribution for all the graphs in $\mathcal{G}$. Fixing a time interval $[t_1,t_2]$, we consider $\overline{P} =\frac{1}{t_2-t_1}\left(P^{(t_1)}  + \cdots + P^{(t_2)}\right)$.  We consider time intervals for which $\overline{P}$ is irreducible. Notice that since the $P^{(i)}$'s are strongly aperiodic and reversible with respect to $\pi$, so is $\overline{P}$. Therefore, we can always assume that $\overline{P}$ is ergodic and has a unique stationary distribution $\pi$, unlike the individual $P^{(i)}$'s.


For simplicity, we assume in our proofs that each graph in $\mathcal{G}$ has the same number of edges $m$. Our results, however, also hold for sequences of graphs with different edges densities.

Notice that, by the detailed balance condition, if $u \sim_i v$ for some step $i$, $\pi(u) / \pi(v) = d_v / d_u$, where $d_u$ and $d_v$ are, respectively, the (time-independent) degrees of $u$ and $v$\footnote{it may happen that $u$ is isolated in some round $i$, leading to $u$ having degree $0$ in that round. However, in that case, $u$ can be safely ignored when computing $\calE_{P^{(i)}}$. Hence, because the stationary distribution is always the same and so is the number of edges, we may assume that the degree of $u$ is always $d_u$}. In particular, this means there exists some $\alpha_u \ge 0$, which is independent from $t$, such that $\pi(u) = \alpha_u d_u/2m$ and $\pi(v) = \alpha_u d_v/2m$.

\begin{lemma}
\label{lem:imp}
Let $p^{(0)}$ be an arbitrary initial probability distribution, and $\rho^{(0)} = p^{(0)}/\pi$. Suppose that for some $t \ge 1$ and $u \in V$, $| \rho^{(t)}(u) - \rho^{(0)}(u) | \ge \eps > 0$. Then,
\[
\var_{\pi}{\rho^{(0)}} - \var_{\pi}{\rho^{(t)}} \ge \frac{\alpha_u}{4m} \sum_{i=1}^t \sum_{v \sim_i u} \left( \rho^{(i-1)}(u) - \rho^{(i-1)}(v) \right)^2
		\ge \frac{2\eps^2 \pi(u)}{t}.
\]
\end{lemma}

We are now able to relate 
$\var_{\pi} \rho^{(0)}  - \var_{\pi} \rho^{(t)}$ to $\mathcal{E}_{\overline{P}}(\rho^{(0)},\rho^{(0)})$. The proof of the next theorem works roughly as follows. We divide the vertices in two classes: $U$ contains all the vertices for which there exists an $1 \le i \le t-1$ such that $\rho^{(i)}(u)$ differs \emph{significantly} from $\rho^{(0)}(u)$, while $V \setminus U$ contains the rest. We then use \lemref{imp} to lower bound the contribution given by vertices in $U$ to $\var_{\pi} \rho^{(0)}  - \var_{\pi} \rho^{(t)}$. Since for $u \not\in U$, $\rho^{(i)}(u)$ has not changed much from $\rho^{(0)}(u)$, we can instead directly lower bound its contribution to $\var_{\pi} \rho^{(0)}  - \var_{\pi} \rho^{(t)}$ just looking at its contribution to $\mathcal{E}_{\overline{P}}(\rho^{(0)},\rho^{(0)})$.

\begin{theorem}
\label{thm:average}
Given a time interval of length $t$ labelled $[1,t]$, let $\overline{P} =\frac{1}{t}(P^{(1)} + P^{(2)} + \cdots + P^{(t)})$ with spectral gap $\lambda(\overline{P})$. Then, for any initial probability distribution $p^{(0)}$ with likelihood $\rho^{(0)} = p^{(0)}/\pi$, it holds that
\[
	\var_{\pi} \rho^{(0)}  - \var_{\pi} \rho^{(t)} \ge \frac{1}{15t} \mathcal{E}_{\overline{P}}(\rho^{(0)},\rho^{(0)}) \ge \frac{\lambda(\overline{P})}{15t}.
\]
\end{theorem}

We remark we do not know if the dependency of $t$ in the bound of $\thmref{average}$ (which appears as a result of an application of the Cauchy-Schwarz inequality) is tight, or even if any dependency on $t$ is needed at all. 

From \thmref{average} it is easy to derive the following corollary:
\begin{corollary}
\label{cor:average}
Given a lazy random walk on a sequence $\calG$ of graphs with transition matrices $\{P^{(i)}\}_{i=1}^{\infty}$ such that (1) there exists $\pi$ which is a stationary distribution for any $P^{(i)}$; (2) a time-window $t \ge 0$ such that, for any $i \ge 0$, $\overline{P}^{[i,i+t]} =\frac{1}{t}(P^{(i)} + P^{(i+1)} + \cdots + P^{(i+t)})$ is ergodic and has spectral gap $\lambda\bigl(\overline{P}^{[i,i+t]}\bigr) \ge \lambda > 0$. Then, $
\tmix(\calG) = O\bigl(\frac{t^2\log(1/\pi_*)}{\lambda}\bigr)$.
\end{corollary}

To highlight the applicability of Corollary \ref{cor:average}, consider a sequence of connected graphs $\mathcal{G}$ with time-independent stationary distribution $\pi$ in which, for any interval of $t$ consecutive steps and subset of vertices $A$, there exists a transition matrix $P^{(i)}$ of a graph in the interval such that $\Phi_{P^{(i)}}(A) \ge \phi$. Then, $\Phi\left(\overline{P}\right) \ge \phi/t$ and $\lambda\left(\overline{P}\right) \ge \phi^2/t^2$. Hence, Corollary \ref{cor:average} gives us that $\tmix(\mathcal{G}) = O(t^3 \log{n}/ \phi^2 )$.

Another natural question is whether our condition on the stationary distribution being fixed could be relaxed. This question is answered negatively by the following result:

\begin{proposition}\label{pro:nomixing}
For any $t=\omega(\log n)$, there is a sequence of connected $n$-vertex bounded-degree expander graphs $\calG =\{G^{(i)}\}_{i=1}^{\infty}$ and a constant $c > 0$ so that $p_{u,v}^{(t)} \geq n^{-1+c}$ for some vertices $u$ and $v$.
\end{proposition}

In \secref{worsthit} and \secref{twopluseps} we have shown that the behaviour of a lazy random walk on a sequence of \emph{connected} graphs with the same stationary distribution is comparable to the behaviour of random walks on static graphs, at least regarding mixing and hitting times. When the graphs are disconnected, however, the behaviour of random walks on dynamic graphs becomes more complicated. \thmref{average} shows that, if every $t$ steps the average of the transition matrices applied in those steps is irreducible and strongly aperiodic with stationary distribution $\pi$, then the random walk will converge to $\pi$. However, $\pi$ can be highly imbalanced and, as a result, mixing and hitting can be exponential in $t$ and the number of vertices $n$. The next lemma shows an example of this behaviour.

\begin{proposition}
\label{pro:nohitting}
There is a sequence of $n$-vertex  bounded-degree graphs $\calG =\{G^{(i)}\}_{i=1}^{\infty}$ with transition matrices $\{P^{(i)}\}_{i=1}^{\infty}$  and a probability distribution $\pi$ such that (1) for any $i$, $\pi$ is stationary for $P^{(i)}$; (2) the average transition matrix $\overline{P}$ of any $4n$ consecutive steps is ergodic; (3) for any $t \ge 0$ there are two vertices $u,v$ such that $p_{u,v}^{[0,t]} \leq 2^{-(n/4)-2}$. Moreover, $\tmix(\calG) = O(\poly(n))$, while $\thit(\calG) = 2^{\Omega(n)}$. There is also a sequence $\calG'$ satisfying (1), (2), and (3) such that $\tmix(\calG) = 2^{\Omega(n)}$.
\end{proposition}

%% file: cutsets.tex
\section{Bounds in terms of average edge connectivity}\label{sec:cutsets}

Recall that in Section~\ref{sec:twopluseps} we proved several bounds which hold for graphs with sufficient expansion for small sets of vertices. Following a different direction, we now derive bounds on commute times for random walks on $d$-regular \emph{static} graphs based on \emph{average} connectivity measures (see the end of Section $2$ for some basic relations between the commute time and hitting time). We assume $G = (V,E)$ is a connected, undirected and static graph with vertex set $V = \{1,\dots,n\}$ and $m$ edges. We denote with $P$ the transition matrix of a lazy random walk on $G$ and $\pi$ its stationary distribution. Given $A, B$, the probability flow between $A$ and $B$ is defined as $\sum_{u \in A} \sum_{v \in B} \pi(u) P(u,v)$. The edge boundary of $A$, denoted with $\partial A$, is the set of edges with one endpoint in $A$ and one in $V \setminus A$. For ease of notation we define $[i] = \{1,\dots,i\}$. Also recall that we denote with $C_{st}$ the expected commute time between $s$ and $t$.  We will use the following variational characterisation of the average commute time (see Aldous and Fill, \cite[Theorem 3.36]{aldousFill}):
\begin{equation}
\label{eq:var_hit}
C_{st} = \max_{g \colon V \to \R} \{ 1/\calE_P(g,g) \colon 0 \le g \le 1, g(s) = 0, g(t) = 1 \}.
\end{equation}

\begin{lemma}\label{lem:general}
For any graph $G = (V,E)$ and $s,t \in V$, there exists a labelling of the vertices from $1$ to $n$ such that
\[
C_{st} \le 2m \sum_{j=1}^{n-1}\frac{1}{|\partial [j]|}.
\]
Furthermore, by considering the reversal of the labelling, we can also conclude that
$
C_{st} \le 4m \sum_{j=1}^{n/2} \frac{1}{|\partial [j]|}.
$
\end{lemma}

Note that the well-known Nash-Williams's inequality \cite[Proposition 9.15]{levinPeres} gives a very similar lower bound: it states that for every set of edge-disjoint cutsets separating $s$ from $t$, $\{E_1,E_2,\ldots,E_k\}$,
$C_{st} \geq 2m \sum_{j=1}^{k} \frac{1}{|E_j|}$. Note that in our upper bound however, the cutsets $\partial[j]$ are in general not edge-disjoint.

Another interesting consequence can be obtained by expressing the above upper bound in terms of some variant of conductance profile. To this end, assume $G$ is $d$-regular and for any $1 \leq k \leq n/2$, let $\Phi_{k} = \min_{|S| \subseteq V, |S|=k} \frac{|E(S, V \setminus S|}{d \cdot |S|}$. Then,
\[
 C_{st} \leq 4 nd \sum_{j=1}^{n/2} \frac{1}{\Phi_k \cdot d j} = 4n \sum_{j=1}^{n/2} \frac{1}{\Phi_j \cdot j} \leq 4n \log n \cdot \frac{1}{\Phi}.
\]
Also let us note that if we replace by $C_{st}$ the maximum commute time, and use the Random Target Lemma~\cite{aldousFill} as a lower bound, we obtain the inequality
\begin{align}
 \sum_{k=2}^n \frac{1}{1-\lambda_k} \leq 4n \sum_{j=1}^{n/2} \frac{1}{\Phi_j \cdot j}. \label{eq:interesting}
\end{align}
Note that for the graph $K_{n/2} \times K_2$, $\Phi_j \sim \frac{n/2-j}{n/2}$, and the right hand side is $O(n \log n)$, whereas the right hand side is $\Omega(n)$.
Hence this inequality is almost tight for certain graphs, and could be seen as some ``average version'' of Cheeger's inequality.

\begin{remark}
As we will prove (see appendix, \lemref{connected}), the lemma above holds even for a labelling such that the subgraph induced by $[i]$ is connected for every $1 \leq i \leq n$.
\end{remark}

\subsection{Commute times and edge-connectivity}
We now apply \lemref{general} to obtain a bound on commute times that depends on the edge-connectivity of the graph, improving a result by Aldous and Fill \cite[Proposition~6.22]{aldousFill}.


\begin{lemma}\label{lem:edgeconn}
Let $G=(V,E)$ be any graph with minimum degree $\delta$ so that any subset $S \subseteq V$ with $1 \leq |S| \leq n-1$ satisfies $|\partial S| \geq \rho$ (in other words, $G$ has edge-connectivity at least $\rho$). Then we have
\[
 \sum_{i=1}^{n-1} \frac{1}{|\partial[i]|} = O(n/\delta^2 \cdot \log \delta + n/(\rho \delta)).
\]
\end{lemma}

\begin{theorem}\label{thm:connectivity}
Let $G=(V,E)$ be any graph with minimum degree $\delta$, average degree $\overline{d}$ and edge-connectivity $\rho$. Then any commute time is bounded by $O(n^2 \overline{d} \cdot ( \frac{\log \delta}{\delta^2} + \frac{1}{\delta \rho} ) )$.
\end{theorem}
\begin{proof}
Simply combining \lemref{edgeconn} with~\lemref{general} yields the statement of the theorem.
\end{proof}

We remark that Aldous and Fill~\cite[Proposition~6.22]{nonrevFill}
proved that for any graph $G$ with average degree $\overline{d}$ which is $\rho$-edge-connected, the maximum commute time is $O(n^2 \overline{d} \cdot \rho^{-3/2})$. They also mention that if the graph is $\Omega(d)$-edge-connected, they obtain a bound of $O(n^2 \cdot d^{-1/2})$. For this case of maximal edge-connectivity, $\rho=\Theta(d)$, 
our bound is considerably better than the one by Aldous and Fill, and, modulo the $\log d$-factor, gives also the correct dependency on $d$.
Furthermore, since the edge-connectivity $\rho$ satisfies $\rho \leq \delta \leq d$, it is easy to verify that our bound is never worse than the bound in Aldous and Fill. In fact, as soon as $\delta \rightarrow \infty$, our upper bound will be asymptotically smaller than the bound by Aldous and Fill.

\begin{remark}[Proved in \lemref{optimalconn}]
For any pair of $\rho$ and $d$ there is a graph matching the upper bound in Theorem~\ref{thm:connectivity} up to a factor of $O(\log d)$.
\end{remark}
\newpage


%% file: appendix.tex
\section{Omitted proofs from \secref{worsthit}}
To prove the results in \secref{worsthit} we will need the following technical lemma:
\begin{lemma}
\label{lem:ballsize}
Let $G$ be a graph with $n$ vertices and minimum degree $\delta$. Then for any integer $1 \le x \le n$, the number of vertices reachable in $x$ hops from any vertex is at least $\min\{\delta \cdot x/3, n\}$.
\end{lemma}
\begin{proof}
The statement is trivially satisfied for $x=1$. For $x > 1$, let $v$ be a vertex at distance $x$ from a vertex $u$. Since $v$ has minimum degree $\delta$, there exist at least $\delta + 1$ vertices at distance $x$, $x-1$, or $x+1$ from  $u$ (unless we can already reach all the vertices in the graphs in $x+1$ steps from $u$). Summing up the number of vertices that we can reach in $i=0,\dots,x$ steps we obtain the statement.
\end{proof}

\begin{lemma}[\lemref{dirichlet} (restated)]
 Let $P$ be the transition matrix of a lazy random walk on a graph $G \in \mathcal{G}$. Given a probability distribution $\sigma: V \to [0,1]$ with likelihood ratio $f = \sigma/\pi$ such that $\var_{\pi} f = \eps  > 0$, 
\[
\mathcal{E}_{P}(f,f) \gtrsim \max\left\{ \frac{\eps^2}{m^* + 1/(\pi_*^2(1+\eps))},  \frac{\pi_* \eps^2}{ n} \right\}.
\]
 \end{lemma}

 \begin{proof}
Denote with $m$ the number of edges of $G$ and with $\delta$ its minimum degree so that $\pi_* = \min_u \{\pi(u)\} = \delta/m$. Let $x \triangleq \argmax_{y} f(y)$, i.e., $\|f\|_{\infty} = f(x)$. Since $\Ev_{\pi} f = 1$ and $\Ev_{\pi} f^2 = 1+ \eps$, there exists at least one vertex $y$ such that $f(y) < 1$. Take the $y$ closest to $x$ (w.r.t. to the shortest path distance in $G$) satisfying $f(y) \le 1+\eps/2$ and let $\ell$ be the distance between $x$ and $y$. By construction, all the vertices  $z$ in $B = B_{\ell-1}(x)$, the ball of radius $\ell-1$ around $x$, satisfy $f(z) > 1+\eps/2$. Therefore, by Markov's inequality, we have that
 \[
 \pi(B) \le \pi\left(\left\{ z \colon f(z) > 1+\eps/2 \right\}\right) \le \frac{\Ev_{\pi} f}{1+\eps/2} = \frac{1}{1+\eps/2}.
 \]
By \lemref{ballsize}, we also have that $|B| \ge \delta \cdot (\ell-1)/3$. Therefore, 
\[
\frac{\delta \cdot (\ell-1)}{3} \cdot \frac{\delta}{m} \le \pi(B) \le \frac{1}{1+\eps/2},
\]
which implies that $\ell \le \min\{3m/(\delta^2 \cdot (1+\eps/2)) + 1, n/(3\delta)\}$. 

Since $x$ and $y$ are at distance $\ell$, there exists a path $x = u_0, u_1, \dots, u_{\ell} =y$ in $G$. Applying the Cauchy-Schwarz inequality and the bound on $\ell$, we obtain:
\begin{align*}
\mathcal{E}_{P}(f,f) &= \frac{1}{4m} \sum_{u \sim v} \left(f(u)-f(v)\right)^2 \nonumber \\
		&\ge \frac{1}{4m} \sum_{i=0}^{\ell-1} \left(f(u_i) - f(u_{i+1})\right)^2 \nonumber \\
		&\ge \frac{1}{4m \cdot \ell} \left(\sum_{i=0}^{\ell-1} \left(f(u_i) - f(u_{i+1})\right)\right)^2 \nonumber \\
		&= \frac{1}{4m \cdot \ell} \left(f(x) - f(y)\right)^2 \nonumber \\
		& \gtrsim \frac{\eps^2}{m \cdot \ell} \\
		&\gtrsim \max\left\{\frac{\eps^2}{m + m^2/(\delta^2(1+\eps))}, \frac{\delta \eps^2}{m \cdot n} \right\} \nonumber \\
		&\gtrsim \max\left\{ \frac{\eps^2}{m^* + 1/(\pi_*^2(1+\eps))},  \frac{\pi_* \eps^2}{ n} \right\} \nonumber.
\end{align*}
 \end{proof}

\begin{lemma}[\lemref{inftoell2} (restated)]
Let $t_1 < t_2$. Then, for any $u,v \in V$, it holds that
\[
\left|\rho^{[t_1,t_2]}_{v,u} - 1 \right| \le \max\left\{\var_{\pi} \left( P^{(\floor{\frac{t_1+t_2}{2}})} \cdots P^{(t_2)}  \rho^{[t_1,t_1]}_{u,\cdot}\right) ,
				 \var_{\pi} \left(P^{(\floor{\frac{t_1+t_2}{2}-1})} \cdots P^{(1)}\rho^{[t_1,t_1]} _{v,\cdot} \right) \right\}.
\]
\end{lemma}
\begin{proof}
First notice that, for $t_2 > t_1$ and any $u,v \in V$, $\rho^{[t_1,t_2]}_{v,u} = \left\langle \rho^{[t_1,t_1]}_{u,\cdot} ,\rho^{[t_1,t_2]}_{v,\cdot} \right\rangle_{\pi}$.
This can be easily seen by the fact that $\rho^{[t_1,t_1]}_{u,\cdot}$ is equal to $1/\pi(u)$ at $u$ and to zero everywhere else. Therefore, using the fact that, for a transition matrix $P$, $P1 = 1$, and the fact that $\Ev_{\pi} f = 1$ for any likelihood $f$,
\begin{align*}
&\left|\rho^{[t_1,t_2]}_{v,u} - 1\right| \\
		&\qquad=\left|\left\langle \rho^{[t_1,t_1]}_{u,\cdot} - 1,\rho^{[t_1,t_2]}_{v,\cdot} -1 \right\rangle_{\pi} \right| \\
		&\qquad= \left|\left\langle \rho^{[t_1,t_1]}_{u,\cdot} - 1 ,P^{(t_2)} \cdots P^{(t_1)} (\rho^{[t_1,t_1]} _{v,\cdot} - 1) \right\rangle_{\pi} \right| \\
		&\qquad= \left|\langle   P^{(\floor{(t_1+t_2)/2})} \cdots P^{(t_2)}  (\rho^{[t_1,t_1]}_{u,\cdot}-1), 
				P^{(\floor{(t_1+t_2)/2-1})} \cdots P^{(1)}(\rho^{[t_1,t_1]} _{v,\cdot} - 1)\rangle_{\pi}\right|  \\
		&\le \max\left\{\| P^{(\floor{(t_1+t_2)/2})} \cdots P^{(t_2)}  \rho^{[t_1,t_1]}_{u,\cdot} - 1 \|_\pi^2, 
				 \|P^{(\floor{(t_1+t_2)/2-1})} \cdots P^{(1)}\rho^{[t_1,t_1]} _{v,\cdot} - 1\|_{\pi}^2\right\},
\end{align*}
where the third equality follows from multiple applications of reversibility,  i.e., \eq{selfadjoint}, and the last line by the Cauchy-Schwarz inequality. The statement follows by the definition of $\var_{\pi} f$.
\end{proof}

\begin{theorem}[\thmref{worsthit} (restated)]
Let $\mathcal{G}$ be a sequence of connected graphs with $n$ vertices, the same stationary distribution $\pi$, and at most $m^*$ edges in each graph. Then, for a lazy random walk on $\mathcal{G}$:
\begin{enumerate}\itemsep0pt
\item $\tmix(\mathcal{G}) = O(n/\pi_*)$,
\item $\bigl| \frac{p^{[0,t]}_{u,v}}{\pi_v} - 1 \bigr| \lesssim \frac{m^*}{t} + \frac{1}{\pi_* \sqrt{t}}$, simplifying to $\bigl| \frac{p^{[0,t]}_{u,v}}{\pi_v} - 1 \bigr| \lesssim \frac{n}{\sqrt{t}}$  if all the graphs in $\mathcal{G}$ are $d$-regular,
\item $\thit(\mathcal{G}) = O(n \log{n}/\pi_*)$. Furthermore, if the graphs in $\mathcal{G}$ are $d$-regular,  $\thit(\mathcal{G}) = O(n^2)$.
\end{enumerate} 
\end{theorem}
\begin{proof} 
Thanks to \lemref{dirichlet} we can readily bound the mixing time of a random walk in $\mathcal{G}$. Let $u$ be an arbitrary vertex in $V$ and $\rho^{(t)} = \rho^{[0,t]}_{u,\cdot}$. Let $\eps^{(t)} = \var_{\pi} \rho^{(t)} = \Ev_{\pi} {\rho^{(t)}}^2 - 1$. First notice that $\var_{\pi} \rho^{(0)} = \eps^{(0)} \le 1/\pi_*$. Moreover,  \lemref{dirichlet} implies that in $\bar{t} = O(n/(\eps^{(t)} \pi_*))$ steps $\eps^{(t)}>1$ is halved. More precisely, $\eps^{\left(t+\bar{t}\right)} \le \eps^{(t)}/2$. As a consequence,
\begin{equation}
\label{eq:tmix}
\tmix(\mathcal{G}) \lesssim \frac{n}{\pi_*} \cdot \sum_{i=1}^{-\log{ \pi_*}} 2^{-i} \lesssim \frac{n}{\pi_*}. 
\end{equation}

We now use \lemref{inftoell2} to obtain bounds on the individual probabilities $p^{[0,t]}_{u,v}$. As remarked in \secref{worsthit} before \lemref{inftoell2}, while for time-homogeneous reversible Markov chains there exists a clear relation between the $\ell_{\infty}$ norm of a $t$-step probability distribution and its variance, in our case, since the chain is time-inhomogeneous, this relation doesn't necessarily holds. What \lemref{inftoell2} shows, however, is that after $t$ steps the decrease in the $\ell_{\infty}$ is roughly equivalent to the \emph{worst-case} decrease in the variance after $t/2$ steps. In particular, it essentially shows that $\rho^{[t_1,t_2]}_{v,u} = p^{[t_1,t_2]}_{v,u} / \pi$ is as small as $\var_{\pi} \rho^{\floor{(t_1+t_2)/2}}$ for some initial likelihood ratio $\rho^{(0)}$ and $\floor{(t_1+t_2)/2}$ steps of random walk on graphs appearing in $\mathcal{G}$ (not necessarily in the same order as they appear in $\mathcal{G}$). But any step decreases the variance of $\rho^{(t)}$ at least by the quantity given by \lemref{dirichlet}. More precisely, \lemref{dirichlet} states that, if $\var_{\pi} \rho^{(t)} = \eps^{(t)}$, $\eps^{(t)}$ is halved in $O(m^*/\eps^{(t)} + 1/(\pi_*^2 \eps^{(t)}(1+\eps^{(t)}) ))$ steps. From this and \lemref{inftoell2} it follows that
\begin{equation}
\label{eq:returns}
\left| \frac{p^{[t_1,t_2]}_{v,u}}{\pi_u} - 1 \right| \lesssim \frac{m^*}{t_2-t_1} + \frac{1}{\pi_* \sqrt{t_2-t_1}}.
\end{equation}
Moreover, in the case of a sequence of $d$-regular graphs we can further simplify this bound using the simple fact that $p^{[t_1,t_2]}_{v,u} \le 1/d  + 2^{-(t_2-t_1)}$ to obtain
\begin{equation}
\label{eq:returns_reg}
\left| \frac{p^{[t_1,t_2]}_{v,u}}{\pi_u} - 1 \right| \lesssim \frac{n}{\sqrt{t_2-t_1}}.
\end{equation}

Let us now bound the maximum hitting time. For any $u,v \in V$, let $X^{(t)}_{u,v}$ be a boolean random variable which is equal to $1$ if and only if a walk starting from $v$ visits $u$ at time $t$. Let $Z = \sum_{t=T}^{2T} X_{v,u}$ be a random variable counting the number of times the walk hits $u$ starting from $v$ in a time $T = \Theta(n/\pi_*)$ such that $T \ge 4 \tmix(\mathcal{G})$. We first need to lower bound $\E{Z}$, i.e., the expected number of visits to $v$ in that time interval. Since $T \ge 4 \tmix(\mathcal{G})$, by the relation discussed above between the variance and the individual probabilities in the random walk distribution, we know that for any $t \ge 2\tmix(\mathcal{G})$, $p^{[t_1,t_2]}_{v,u} \ge (8/9) \cdot \pi(u)$. Therefore, we have that $\E{Z} \gtrsim T \cdot \pi(u)$, and 
\begin{equation}
\label{eq:hitupper}
\Pr{Z \ge 1}  = \frac{\E{Z}}{\E{Z \,|\, Z \ge 1}} \gtrsim  \frac{(n/\pi_*) \cdot \pi(u)}{ \max_{T \le \tau \le 2T} \sum_{t=0}^{T} p^{[\tau,\tau+t]}_{u,u}}.
\end{equation}
It remains to upper bound $\E{Z \,|\, Z \ge 1} \le \max_{T \le \tau \le 2T} \sum_{t=0}^{T} p^{[\tau,\tau+t]}_{u,u}$ which corresponds to the expected number of returns to $u$ in the interval $[\tau,\tau+t]$. By \lemref{dirichlet} and \lemref{inftoell2}, we know that $\eps^{(t)} = |p^{[\tau,\tau+t]}_{u,u}/\pi(u) -1|$ is halved every $O(n/(\pi_* \eps^{(t)}))$ steps. Clearly, the sum of the return probabilities in a window of $O(n/(\pi_* \eps^{(t)}))$ steps is $O(\pi(u) \cdot n /\pi_*)$. Since there are at most $O(\log{1/\pi_*})$ such time windows, we have that
\[
\E{Z \,|\, Z \ge 1} \le \max_{T \le \tau \le 2T} \sum_{t=0}^{T} p^{[\tau,\tau+t]}_{u,u} \lesssim \pi(u) \cdot (n/\pi_*) \cdot \log{1/\pi_*}.
\]
Together with \eq{hitupper}, we obtain that $\Pr{Z \ge 1} = \Omega(1/\log{n})$, which means that starting from $v$, with 
 probability $\Omega(1/\log{n})$, we will hit $u$ in $O(n/\pi_*)$ steps. We can conclude then, since the vertex $v$ was chosen arbitrarily, that the hitting time is $\thit(\mathcal{G}) = O(n\log{n}/\pi_* )$. For regular graphs, the same argument combined with \eq{returns_reg} and the fact that $\pi$ is uniform, leads to a $O(n^2)$ hitting time.
\end{proof}

\section{Omitted proof from \secref{twopluseps}}
\begin{lemma}[\lemref{decrease_eps} (restated)]
Let $G=(V,E)$ be a $d$-regular undirected graph with $|V| = n$ and $d=O(1)$ such that, for any $A \subset V$ with $1 \le |A| \le n/2$,
$
|E(A,V \setminus A)| = \Omega(|A|^{\frac{1}{2} + \eps})
$ for $1/4 \ge \eps \ge 0$. Consider the transition matrix $P$ of a lazy random walk in $G$. Let $\sigma$ be any probability distribution and $f = \sigma/\pi$, where $\pi$ is the uniform distribution.  If $\Ev_{\pi} f^2 = \beta > C$ for a large enough constant $C$, then
\[
\mathcal{E}_P(f,f) \gtrsim  \frac{\beta^{2-2\eps}}{n^{1-2\eps}}.
\]
\end{lemma}
\begin{proof}
We first define a set $K \subset V$ which contains all vertices with large probability mass:
\[
K = \{ u \in V \colon f(u) \ge \beta/10 \} = \{ u \in V \colon \sigma(u) \ge \beta/(10 n) \}.
\]
We first assume that $|K| = \Omega(n/\beta)$ and later we will show how to drop this assumption.

By the isoperimetric property of $G$, we can choose $s  \gtrsim |K|^{\frac{1}{2} + \eps} d^{-2} \gtrsim (n/\beta)^{\frac{1}{2} + \eps} d^{-2} $ vertex-disjoint paths $\mathcal{P}_1,\dots,\mathcal{P}_s$ of length (respectively) $\ell_1,\dots,\ell_s$ such that each path starts in $K$, but all the other vertices visited by the path are outside $K$. Moreover, let $(u_1,v_1),\dots,(u_s,v_s)$ be the pair of starting/ending points of the paths. We claim we can choose the paths so that $f(v_i) < \beta/20$ and $(f(u_i)-f(v_i))^2 \ge  (\beta/20)^2$ for any $1 \le i \le s$. Such paths can be constructed as follows: thanks to the isoperimetric property of $G$ we can choose $\Omega\left(|K|^{\frac{1}{2} + \eps}\right)$ edges that goes from $K$ to its complement. Since any vertex has degree at most $d$, we can pick  $s  \gtrsim |K|^{\frac{1}{2} + \eps} d^{-2} $ such edges that do not share any endpoint. Let $K' \subset V$ be the union of $K$ and the endpoints of these $s$ edges and let $s' = |\{u \in K' \colon f(u) < \beta/20\}|$. Again by the isoperimetric property, we can choose other $(s-s')$ edges from $K'$ to its complement, so that no edge involves any vertex in $K$ (if an edge involves a new, ``unused'' , vertex from $K$, we could have used that edge earlier in the construction of such paths). We can repeat this process till we have constructed a set $K'' \subset V$ containing at least $s$ vertices $v_1,\dots,v_s$ such that $f(v_i) < \beta/20$. Notice that each vertex $v_i$ is

%
 connected by a path $\mathcal{P}_i$ to a vertex $u_i \in K$ and the paths do not share any vertex. Given paths $\mathcal{P}_1,\dots,\mathcal{P}_s$ constructed in this way, we can bound $\mathcal{E}_P(f,f)$ as follows:
\begin{align*}
\mathcal{E}_P(f,f) 	&= \frac{1}{d n} \sum_{u \sim v} (f(u)-f(v))^2 \\
				&\ge \frac{1}{d n}  \sum_{i=1}^s \sum_{\{u,v\} \in \mathcal{P}_i} (f(u)-f(v))^2 \\
				&\ge \frac{1}{d n}  \sum_{i=1}^s \frac{(f(u_i)-f(v_i))^2}{\ell_i} \gtrsim \frac{\beta^2}{dn} \sum_{i=1}^s \frac{1}{\ell_i}.
\end{align*}
Since $\Ev_{\pi} f = 1$, by Markov's inequality there can be at most  $20 n/\beta$ vertices $x$ such that $f(x) \ge \beta/20$. Hence, we have that $\sum_{i=1}^s \ell_i \le 20 n/ \beta$. Moreover, $s \gtrsim (n/\beta)^{\frac{1}{2} + \eps} $ (from now on we drop the dependence on $d$ since $d=O(1)$). By the relation between harmonic and arithmetic mean, $\sum_{i=1}^s \frac{1}{\ell_i} \ge \frac{s^2}{\sum_{i=1}^s \ell_i}$, and
\[
\mathcal{E}_P(f,f) \gtrsim \frac{\beta^2 s^2}{n \sum_{i=1}^s \ell_i} \gtrsim \frac{\beta^{2-2\eps}}{n^{1-2\eps}},
\]
which conclude the proof as long as $|K| = \Omega(n/\beta)$. We now show how to achieve the same result without this assumption on $K$.
We divide the vertices into buckets:
$
  K_i:= \left\{ u \in V \colon f(u) \ge \beta/100 \cdot 2^i  \right\}.
$
Suppose now that for every $K_i$ we have $|K_i| \leq n/\beta \cdot 2^{-2i}$.
Then it follows that:
\[
\Ev_{\pi} f^2 \le \frac{1}{n} \sum_{\substack{u \in V \\ f(u) \leq \beta/100}} f(u)^2 + \frac{1}{n} \sum_{i=1}^{\log_2 n} |K_i| \cdot (\beta/100)^2 \cdot 2^{i}.
\]
Since $\Ev_{\pi} f = 1$, the first term is bounded by $\beta/100$. Furthermore, the second term is also at most a small constant times $\beta$ thanks to our condition on $|K_i|$. Therefore, $\Ev_{\pi} f^2 < \beta$, which contradicts the assumption  $\Ev_{\pi} f^2 = \beta$.
In conclusion, there must exist an index $1 \leq i \leq \log_2 n$ such that
\[
|K_i| \geq n/\beta \cdot 2^{-2i}.
\]
In this case, we consider again the paths to vertices $v$ with $\sigma(v) \leq \beta \cdot 2^i/100$. As argued above, there will be $s \gtrsim {|K_i|}^{1/2+\eps}$ vertex-disjoint paths with lengths $\ell_1, \dots, \ell_s$ such that $\sum_{i=1}^s \ell_i \le 200 n /(\beta\cdot 2^i)$ and the gap between the values taken by $f$ on their extremities is at least $\Omega( \beta \cdot 2^i)$. Following the previous argument we have that
\begin{align*}
\mathcal{E}_P(f,f) 	&\gtrsim \frac{2^{2i}\beta^2}{n} \sum_{i=1}^s \frac{1}{\ell_i} \ge  \frac{2^{2i} \beta^2 s^2}{n \sum_{i=1}^s \ell_i}  \gtrsim \frac{2^{i-4\eps}\beta^{2-2\eps}}{n^{1-2\eps}} \gtrsim \frac{\beta^{2-2\eps}}{n^{1-2\eps}}.
\end{align*}

\end{proof}


\begin{theorem}[\thmref{twopluseps} (restated)]
Let $\mathcal{G}=\{G^{(t)}\}_{t=1}^{\infty}$ be a sequence of $n$-vertex graphs such that each $G^{(t)}$ is regular, has bounded degree, and satisfies the following isoperimetric condition: there exists $\eps \in [0,1/4]$ such that, for any subset of vertices $A$ with $1 \le |A| \le n/2$, $|E(A,V \setminus A)| = \Omega(|A|^{\frac{1}{2} + \eps})$. Then,
\begin{enumerate}\itemsep0pt
\item $\tmix(\mathcal{G}) = O(n^{1-2\eps})$,
\item $\bigl| \frac{p^{[0,t]}_{u,v}}{v} - 1 \bigr| = O\left(\frac{1}{t^{1+2\eps}}\right)$,
\item $\thit(\mathcal{G}) = O(n)$ if $\eps>0$, $\thit(\mathcal{G}) = O(n\log{n})$ if $\eps=0$.
\end{enumerate} 
\end{theorem}
\begin{proof}
We start deriving bounds on mixing and $t$-step probabilities. As shown by \lemref{inftoell2}, $\rho^{[\tau,\tau+t]}_{v,v} = p^{[\tau,\tau+t]}_{v,v}/ \pi(v)$ decreases at a rate which is proportional to the rate given by \lemref{decrease_eps}. In particular, $\beta= \rho^{[\tau,\tau+t]}_{v,v}$ gets halved after $O(n/\beta)^{1-2\eps}$ steps, which implies that $p^{[\tau,\tau+t]}_{v,v} \le \frac{1}{n} + O\left(\frac{1}{t^{1+2\eps}}\right)$ and $\tmix(\mathcal{G}) = O(n^{1-2\eps})$.


We now want to bound the expected hitting time from a vertex $u$ to a vertex $v$.  Let $\pi$ be the uniform distribution. As in the proof of \thmref{worsthit}, we introduce a variable $Z$ which counts the number of visits to $v$ in some time window of length $\Theta(n)$. More precisely, let $T = c \cdot n$ for some large enough constant $c$. We define $Z = \sum_{t=T}^{2T} X_{u,v}^{(t)}$ where $X_{u,v}^{(t)}$ is a Boolean variable that is equal to $1$ if and only if the walk, which started at $u$, is at vertex $v$ at time $t$. Since $T \ge \tmix(\mathcal{G})$, $\E{Z} \gtrsim T \cdot \pi(v) = \Omega(1)$. We then want to bound $\E{Z | Z \ge 1} \le \max_{T \le \tau \le 2T} \sum_{t=0}^T p^{[\tau,\tau+t]}_{v,v}$. We first consider the case when $\eps > 0$. From the bound on $t$-step probabilities derived above, we have $\sum_{t=0}^T p^{[\tau,\tau+t]}_{v,v} = O(1)$.
Therefore, $\Pr{Z \ge 1}  = {\E{Z}}/{\E{Z | Z \ge 1}} = \Omega(1)$, from which it follows that $\thit(\mathcal{G}) = O(n)$.

For the case when $\eps = 0$, instead,  we have that $\E{Z | Z \ge 1} \le \max_{T \le \tau \le 2T} \sum_{t=0}^T p^{[\tau,\tau+t]}_{v,v} = O(\log{n})$. Therefore, $\Pr{Z \ge 1}  = {\E{Z}}/{\E{Z | Z \ge 1}} = \Omega(1/\log{n})$. Since we have a $\Omega(1/\log{n})$ probability to hit $v$ in $O(n)$ steps, we just need to repeat the process $O(\log{n})$ times to obtain a constant probability of hitting $v$, which result in $\thit(\mathcal{G}) = O(n\log{n})$.  
\end{proof}

\section{Omitted proofs from \secref{average}}

\begin{lemma}[\lemref{imp} (restated)]
Let $p^{(0)}$ be an arbitrary initial probability distribution, and $\rho^{(0)} = p^{(0)}/\pi$. Suppose that for some $t \ge 1$ and $u \in V$, $| \rho^{(t)}(u) - \rho^{(0)}(u) | \ge \eps > 0$. Then,
\[
\var_{\pi}{\rho^{(0)}} - \var_{\pi}{\rho^{(t)}} \ge \frac{\alpha_u}{4m} \sum_{i=1}^t \sum_{v \sim_i u} \left( \rho^{(i-1)}(u) - \rho^{(i-1)}(v) \right)^2
		\ge \frac{2\eps^2 \pi(u)}{t}.
\]
\end{lemma}
\begin{proof}
First notice that, for any $1 < i \le t$,
\[
p^{(i)}(u) = \frac{p^{(i-1)}(u)}{2} + \sum_{v \sim_i u} \frac{p^{(i-1)}(v)}{2d_v} = \frac{1}{2}\left(p^{(i-1)}(u) +  \sum_{v \sim_i u} \frac{\alpha_u}{2m} \rho^{(i-1)}(v)\right),
\]
where we used the fact that $\pi(v) = \alpha_u d_v/2m$ for $v \sim_i u$. Therefore, subtracting $p^{(i-1)}(u)$ to both sides,
\begin{align*}
p^{(i)}(u) - p^{(i-1)}(u) &= \frac{1}{2} \sum_{v \sim_i u}\left(  \frac{\alpha_u}{2m} \rho^{(i-1)}(v) - \frac{p^{(i-1)}(u)}{d_u}\right) \\
	&= \frac{\alpha_u}{4m} \sum_{v \sim_i u}\left( \rho^{(i-1)}(v) -  \rho^{(i-1)}(u)\right).
\end{align*}
Taking the absolute value and applying Cauchy-Schwarz, we obtain that
\begin{align*}
\left| p^{(i)}(u) - p^{(i-1)}(u) \right| &\le \frac{\alpha_u}{4m} \sum_{v \sim_i u}\left| \rho^{(i-1)}(v) -  \rho^{(i-1)}(u)\right| \\
	&\le \frac{\alpha_u}{4m} \sqrt{d_u \sum_{v \sim_i u}\left( \rho^{(i-1)}(v) -  \rho^{(i-1)}(u)\right)^2}
\end{align*}
Summing over all $0 < i \le t$,
\begin{align*}
\epsilon 	&\le | \rho^{(t)}(u) - \rho^{(0)}(u) | = \frac{1}{\pi(u)} | p^{(t)}(u) - p^{(0)}(u) | = \frac{1}{\pi(u)} \left| \sum_{i=1}^t (p^{(i)}(u) - p^{(i-1)}(u)) \right| \\
		&\le \frac{\alpha_u}{4 \pi(u) m} \sum_{i=1}^t \sqrt{d_u \sum_{v \sim_i u}\left( \rho^{(i-1)}(v) -  \rho^{(i-1)}(u)\right)^2} \\ 
		&\le  \frac{1}{2} \sqrt{\frac{t}{d_u} \cdot \sum_{i=1}^t \sum_{v \sim_i u} \left( \rho^{(i-1)}(v) -  \rho^{(i-1)}(u)\right)^2}
\end{align*} 
Therefore, by \eq{mihai},
\begin{align*}
\var_{\pi}{\rho^{(0)}} - \var_{\pi}{\rho^{(t)}} &= \sum_{i=1}^t \left(\var_{\pi}{\rho^{(i-1)}} - \var_{\pi}{\rho^{(i)}}\right)\\ 
				&\ge \sum_{i=1}^t \calE_{P^{(i)}}(\rho^{(i-1)},\rho^{(i-1)}) \\
				&\ge \sum_{i=1}^t \sum_{v \sim_i u}  \frac{\pi(u)}{2d_u} \left( \rho^{(i-1)}(u) - \rho^{(i-1)}(v) \right)^2 \\
				&\ge \frac{\alpha_u}{4m} \sum_{i=1}^t \sum_{v \sim_i u} \left( \rho^{(i-1)}(u) - \rho^{(i-1)}(v) \right)^2 \\
				&\ge \frac{\eps^2 \alpha_u d_u}{mt} \\ 
				&= \frac{2\eps^2 \pi(u)}{t}.
\end{align*}
\end{proof}

\begin{theorem}[\thmref{average} (restated)]
Given a time interval of length $t$ labelled $[1,t]$, let $\overline{P} =\frac{1}{t}(P^{(1)} + P^{(2)} + \cdots + P^{(t)})$ with spectral gap $\lambda(\overline{P})$. Then, for any initial probability distribution $p^{(0)}$ with likelihood $\rho^{(0)} = p^{(0)}/\pi$, it holds that
\[
	\var \rho^{(0)}  - \var \rho^{(t)} \ge \frac{1}{15t} \mathcal{E}_{\overline{P}}(\rho^{(0)},\rho^{(0)}) \ge \frac{\lambda(\overline{P})}{15t}.
\]
\end{theorem}
\begin{proof}
Recall that 
\[
\mathcal{E}_{\overline{P}}(\rho^{(0)},\rho^{(0)}) = \frac{1}{t} \sum_{i=1}^t \mathcal{E}_{P^{(i)}} (\rho^{(0)},\rho^{(0)}) 
		= \frac{1}{4  m  t} \sum_{i=1}^t \sum_{u \in V} \sum_{v \sim_i u} \alpha_u \left(\rho^{(0)}(u) - \rho^{(0)}(v) \right)^2.
\]
Denote with $\mathcal{E}_u$ the contribution given by $u \in V$ to $\mathcal{E}_{\overline{P}}(\rho^{(0)},\rho^{(0)})$ normalised by $\pi(u)$, i.e.,
\[
\mathcal{E}_u \triangleq \frac{1}{4 t d_u} \sum_{i=1}^t \sum_{v \sim_i u} \left(\rho^{(0)}(u) - \rho^{(0)}(v) \right)^2.
\]
Let $U \in V$ the subset of vertices $u$ for which there exists an $1 \le i \le t$ such that $\left(\rho^{(i)}(u) - \rho^{(0)}(u) \right)^2$ is greater than $\mathcal{E}_u/30$:
\[
U \triangleq \{ u \in V \colon \exists \, 1 \le i \le t \text{ s.t. } |\rho^{(i)}(u) - \rho^{(0)}(u)| > \sqrt{\mathcal{E}_u/30} \}.
\]

We consider the contribution to $\var_{\pi}{\rho^{(0)}} - \var_{\pi}{\rho^{(t)}}$ given by vertices that belong to $U$ and vertices that do not separately:
\begin{align}
\var_{\pi}{\rho^{(0)}} - \var_{\pi}{\rho^{(t)}} \ge & \frac{1}{4m} \sum_{u \in U} \sum_{i=1}^t  
				\sum_{v \sim_i u} \alpha_u \left( \rho^{(i-1)}(u) - \rho^{(i-1)}(v) \right)^2 \nonumber \\ 			
	&+ \frac{1}{4m} \sum_{u \not\in U} \sum_{i=1}^t  \sum_{v \sim_i u} \alpha_u \left( \rho^{(i-1)}(u) - \rho^{(i-1)}(v) \right)^2. \label{eq:sum0}
\end{align}

We begin with the first part of the sum. By \lemref{imp}, the definition of $U$, and $\pi(u) = \alpha_u d_u/(2m)$, we have that
\begin{align}
\frac{1}{4m} \sum_{u \in U} \sum_{i=1}^t  \sum_{v \sim_i u} &\alpha_u \left( \rho^{(i-1)}(u) - \rho^{(i-1)}(v) \right)^2 \nonumber\\
			&\ge \frac{1}{15t} \sum_{u \in U} \mathcal{E}_u \cdot  \pi(u) \nonumber\\
			&= \frac{1}{15 \cdot 4 m \cdot t^2} \sum_{u \in U} \sum_{i=1}^t \sum_{v \sim_i u} \alpha_u \left(\rho^{(0)}(u) - \rho^{(0)}(v) \right)^2. \label{eq:sum1}
\end{align}

We now look at the second part of the sum. Using the inequality $(a+b+c)^2 \ge a^2/2 - 6b^2 - 6c^2$ we obtain that
\begin{align*}
\frac{1}{4m} \sum_{u \not\in U} &\sum_{i=1}^t  \sum_{v \sim_i u} \alpha_u \left( \rho^{(i-1)}(u) - \rho^{(i-1)}(v) \right)^2 \\
	&\ge \frac{1}{4m} \sum_{u \not\in U} \sum_{i=1}^t  \sum_{v \sim_i u}
	\alpha_u \left( \rho^{(i-1)}(u) - \rho^{(0)}(u) + \rho^{(0)}(u) - \rho^{(0)}(v) + \rho^{(0)}(v) - \rho^{(i-1)}(v) \right)^2 \\
	&\ge \frac{1}{8m} \sum_{u \not\in U} \sum_{i=1}^t  \sum_{v \sim_i u}
			\alpha_u \left(\rho^{(0)}(u) - \rho^{(0)}(v) \right)^2 
			- \frac{3}{m} \sum_{u \not\in U} \sum_{i=1}^t  \sum_{v \sim_i u}
	\alpha_u \left(\rho^{(0)}(u) - \rho^{(i-1)}(u) \right)^2.
\end{align*}
By the definition of $U$, it follows that 
\begin{align*}
\frac{3}{m} \sum_{u \not\in U} \sum_{i=1}^t  \sum_{v \sim_i u} \alpha_u \left(\rho^{(0)}(u) - \rho^{(i-1)}(u) \right)^2 
	&\le \frac{1}{10m} \sum_{u \not\in U} \sum_{i=1}^t  \sum_{v \sim_i u} \alpha_u \mathcal{E}_u\\
	&=\frac{1}{10 m} \sum_{u \not\in U} \sum_{i=1}^t  \sum_{v \sim_i u} \alpha_u \left(\rho^{(0)}(u) - \rho^{(0)}(v) \right)^2.
\end{align*}
Putting together the last two sums, we have that
\begin{align}
\frac{1}{4m} \sum_{u \not\in U} \sum_{i=1}^t  \sum_{v \sim_i u} &\alpha_u \left( \rho^{(i-1)}(u) - \rho^{(i-1)}(v) \right)^2 \ge
\frac{1}{40m} \cdot \sum_{u \not\in U} \sum_{i=1}^t  \sum_{v \sim_i u} \alpha_u \left( \rho^{(0)}(u) - \rho^{(0)}(v) \right)^2. \label{eq:sum2}
\end{align}

Finally, combining \eq{sum0} with \eq{sum1} and \eq{sum2}, it follows that
\begin{align*}
\var_{\pi}{\rho^{(0)}} - \var_{\pi}{\rho^{(t)}} 
		&\ge 	\frac{1}{15 \cdot 4m \cdot t^2} \sum_{u \in U} \sum_{i=1}^t \sum_{v \sim_i u} \alpha_u \left(\rho^{(0)}(u) - \rho^{(0)}(v) \right)^2 \\
		&\qquad + \frac{1}{40m} \cdot \sum_{u \not\in U} \sum_{i=1}^t  \sum_{v \sim_i u} \alpha_u \left( \rho^{(0)}(u) - \rho^{(0)}(v) \right)^2 \\
		&\ge \frac{1}{15 \cdot 4m \cdot t^2} \sum_{u \in V} \sum_{i=1}^t \sum_{v \sim_i u} \alpha_u \left(\rho^{(0)}(u) - \rho^{(0)}(v) \right)^2 \\
		&= \frac{1}{15t} \mathcal{E}_{\overline{P}}(\rho^{(0)},\rho^{(0)}) \\
		&\ge \frac{\lambda(\overline{P})}{15t},
\end{align*}
where the last inequality follows from the relation between the Dirichlet form and the spectral gap of a transition matrix.
\end{proof}

\begin{proposition}[Proposition \ref{pro:nomixing} (restated)]
For any $t=\omega(\log n)$, there is a sequence of connected $n$-vertex bounded-degree expander graphs $\calG =\{G^{(i)}\}_{i=1}^{\infty}$ and a constant $c > 0$ so that $p_{u,v}^{(t)} \geq n^{-1+c}$ for some vertices $u$ and $v$.
\end{proposition}

\begin{proof}
Assume that $p^{(0)}=(1/n,\ldots,1/n)$ is the initial distribution, and $p^{(i)}=p P^{(1)} P^{(2)} \cdots P^{(i)}$ will be the resulting distribution after $i$ steps. Then select a sets of size $S_1 \subseteq S_0=V$ of size $|S_1|=|S_0|/10$ and connect every vertex in $S_1$ to $6$ other vertices in $S_0 \setminus S_1$ in a way so that every vertex in $S_0 \setminus S_1$ has degree at most $1$ (so far). Then add a $3$-regular expander graph to all vertices $V$ to ensure that the resulting graph is a bounded-degree (but non-regular) expander. Note that for any vertex $u \in S_1$,
\[
 p^{(1)}(u) \geq \frac{1}{2} p^{(1)}(u) + \sum_{v \in S_1, v \sim_{1} u}  \frac{1}{8} \cdot p^{(0)}(v) \geq \frac{10}{8} \cdot \frac{1}{n},
\]
since every vertex $v \in S_1, v \sim_{1} u$ has degree $1+3=4$. 

Now proceed by induction. So assume the existence of a subset $S_i$ such that $p^{(i)}(u) \geq (10/8)^i \cdot \frac{1}{n}$ and $|S_i| \geq 10^{-i} \cdot n$. Then again, we choose an arbitrary set $S_{i+1} \subseteq S_i$ of size $|S_{i+1}| \geq |S_i|/10$ and proceed with the same construction as above, replacing $S_0$ by $S_i$ and $S_1$ by $S_{i+1}$. With this substitution we conclude
\[
 p^{(i)}(u) \geq \frac{1}{2} p^{(i)}(u) + \sum_{v \in S_i, v \sim_{i} u}  \frac{1}{8} \cdot p^{(i-1)}(v) \geq \frac{10}{8} \cdot \left( \frac{10}{8} \right)^{i-1} \cdot \frac{1}{n} = \left( \frac{10}{8} \right)^{t} \cdot \frac{1}{n}.
\]
Thus we can proceed inductively as long as $|S_i| \geq 1$, which holds as long as $i \leq c' \cdot \log n$ for some constant $c' > 0$. To obtain the statement of theorem for values of $t=\Omega(\log n)$, simply pad $t-c \log n$ many $3$-regular expander graphs to the beginning of the sequence $\calG$, which will leave the probability distribution unchanged until step $t-c' \log n$. Thus for $p^{(0)}$ being the uniform distribution, we can infer that $p_{u,v}^{(t)} \geq n^{-1+c}$ for $c > 0$ being a constant, $t=\omega(\log n)$ and proper vertices $u,v \in V$.
\end{proof}

\begin{proposition}[Proposition \ref{pro:nohitting} (restated)]
There is a sequence of $n$-vertex  bounded-degree graphs $\calG =\{G^{(i)}\}_{i=1}^{\infty}$ with transition matrices $\{P^{(i)}\}_{i=1}^{\infty}$  and a probability distribution $\pi$ such that (1) for any $i$, $\pi$ is stationary for $P^{(i)}$; (2) the average transition matrix $\overline{P}$ of any $4n$ consecutive steps is ergodic; (3) for any $t \ge 0$ there are two vertices $u,v$ such that $p_{u,v}^{[0,t]} \leq 2^{-(n/4)-2}$. Moreover, $\tmix(\calG) = O(\poly(n))$, while $\thit(\calG) = 2^{\Omega(n)}$. There is also a sequence $\calG'$ satisfying (1), (2), and (3) such that $\tmix(\calG) = 2^{\Omega(n)}$.
\end{proposition}
\begin{proof}
W.l.o.g.~let $n$ be a multiple $4$.
Let $V$ be the vertex set of $\calG$. Divide $V$ into $n/4$ buckets of size $4$ each, and label the buckets $V_1,V_2,\ldots,V_{k}$ with $k=n/4$. In the first $6$ steps, the graphs $G^{(1)},G^{(2)},\ldots,G^{(6)}$ will only have edges between $V_1$ and $V_2$, and in the next $6$ steps, there will be only edges between $V_2$ and $V_3$ and so on. Finally we apply the same scheme periodically, but in reserved order i.e., $G^{(6(n/4)+i)}=G^{(6(n/4)+1-i)}$ for any $1 \leq i \leq 6(n/4)$. From then on we repeat the same sequence, i.e.,
$G^{(12(n/4)+i)}=G^{(i)}$ for any $i \geq 1$.

In each $G^{(1)}$, with $1 \leq i \leq 6$, we will form a complete bipartite graph between two vertices in $V_1$ and (all) four vertices in $V_2$. The number of different bipartite subgraphs is $\binom{4}{2}=6$, so each $G^{(1)}$ will correspond to exactly one complete bipartite subgraph. 

Note that by construction and the fact that we reversed the order in the second period, it is clear that the union of the graphs $G^{(1)},G^{(2)},\ldots,G_{12(n/4)+1}$ is connected, and since we are considering lazy random walks, the average transition matrix $\overline{P}$ of such sequence of graphs will be ergodic.

Let us now verify that $\pi$ with $\pi(u) = 2^{-i-2}$ for any $u \in V_i$ is a stationary distribution, i.e., does not change over time when applied to the graph sequence. To this end, we simply note that the stationary distribution in a complete bipartite graph is proportional to the respective degree, and since the degrees of $V_i$ and $V_{i+1}$ differ by a factor of two, it follows that every $G^{(1)}$ will keep the distribution unchanged. Therefore $\pi$ is the unique stationary distribution of $\overline{P}$. The bound $\thit(\calG) = 2^{\Omega(n)}$ follows trivially. To bound $\tmix$, notice that $\Phi_{P} = \Omega(1)$, and $\tmix(\calG) = O(\poly(n))$ follows from \thmref{average}.

To obtain a sequence with exponential mixing and hitting times, we construct a sequence $\calG'$ with vertex set $V \cup V'$ where $V'$ is a copy of $V$. Any graph in $\calG'$ corresponds to two disjoint copies of the corresponding graph in $\calG$, but every $3n+1$ steps, we add a perfect matching between the vertices in $V_{n/4}$ and the corresponding vertices in $V_{n/4}'$. It is not difficult to see that  $\tmix(\calG') = 2^{\Omega(n)}$ (and, indeed, if we consider the average transition matrix $\overline{P}$ of  any consecutive $3n + 2$ steps, it will be ergodic but with $\Phi(\overline{P}) = 2^{-\Omega(n)}$).
\end{proof}

\section{Omitted proofs and details from \secref{cutsets}}
The following lemma borrows some ideas from the proof of Proposition 4.42 in the book by Aldous and Fill~\cite{aldousFill}.
\begin{lemma}[Lemma~\ref{lem:general} (restated)]
For any graph $G = (V,E)$ and $s,t \in V$, there exists a labelling of the vertices from $1$ to $n$ such that
\[
C_{st} \le 2m \sum_{j=1}^{n-1}\frac{1}{|\partial [j]|}.
\]
Furthermore, by considering the reversal of the labelling, we can also conclude that
$
C_{st} \le 4m \sum_{j=1}^{n/2} \frac{1}{|\partial [j]|}.
$
\end{lemma}
\begin{proof}
Fix a function $g \colon V \to \R$ such that $0 \le g \le 1$ and $0 = g(1) \le g(2) \le \cdot \le g(n) = 1$ for $s = 0, t=n$.  Then we have
\begin{align*}
\calE_P(g,g) 	&=  \sum_i \sum_{i < k} \pi(i) P(i,k) (g(i) - g(k))^2  \\
			&\ge \sum_i \sum_{i < k} \sum_{i\le j < k} \pi(i) P(i,k) (g(j+1) - g(j))^2 \\
			&= \sum_{j} (g(j+1) - g(j))^2 Q([j], V \setminus [j]).
\end{align*}
But by applying Cauchy-Schwarz we obtain
\begin{align*}
1 = \sum_j  (g(j+1) - g(j)) &=  \sum_j  (g(j+1) - g(j)) \frac{Q([j], V \setminus [j])^{1/2}}{Q([j], V \setminus [j])^{1/2}} \\
			& \le  \sum_{j} (g(j+1) - g(j))^2 Q([j], V \setminus [j]) \sum_j \frac{1}{Q([j], V \setminus [j])} \\
			& \le \calE_P(g,g) \sum_j \frac{1}{Q([j], V \setminus [j])}
\end{align*}
and by \eq{var_hit}
\begin{align*}
C_{st} \le \sum_j \frac{1}{Q([j], V \setminus [j])} = 2m \sum_{j} \frac{1}{|\partial [j]|}.
\end{align*}
\end{proof}

\begin{lemma}[Lemma \ref{lem:edgeconn} (restated)]
Let $G=(V,E)$ be any graph with minimum degree $\delta$ so that any subset $S \subseteq V$ with $1 \leq |S| \leq n-1$ satisfies $|\partial S| \geq \rho$ (in other words, $G$ has edge-connectivity at least $\rho$). Then we have
\[
 \sum_{i=1}^{n-1} \frac{1}{|\partial[i]|} = O(n/\delta^2 \cdot \log \delta + n/(\rho \delta)).
\]
\end{lemma}

Before proving this lemma, as a warmup we recover a well-known hitting time bound for regular graphs using \lemref{general}:
\[
C_{st} \leq 2 |E| \dist(s,v) = O\left(n^2 \cdot \frac{d}{\delta}\right).
\]
Specifically, we will prove that
\begin{align}
 \sum_{j=1}^{n-1}\frac{1}{|\partial[j]|} = O(n/\delta). \label{eq:sum}
\end{align}
Note that whenever there is a $j$ for which
\[
|\partial[j]| \in (1,\delta/4],
\]
then we must have
\[
|\partial[j+1]| \geq \delta/2,
\]
since the vertex $j+1$ has degree $d$. For the same reason, at least $d/2$ edges from $j+1$ go to different endpoints in $\{j+2,\ldots,n\}$. Therefore, for any $1 \leq i \leq \delta/4$ we have
\[
 |\partial[j+i]| \geq \delta/4.
\]
Hence a simple amortisation argument shows that any sequence of $\delta/4$ consecutive terms in the sum of \eq{sum} contribute at most $O(1)$, and overall
\[
 \sum_{j=1}^{n-1}\frac{1}{|\partial [j] |} = O(n/\delta),
\]
as claimed.

\begin{proof}[Proof of Lemma \ref{lem:edgeconn}]
Let $1 \leq j \leq n-\epsilon \delta-1$ be arbitrary. We will now perform an amortised analysis on 
\begin{align*}
 \Delta_{j} &= \sum_{i=j}^{j+\epsilon \cdot \delta} \frac{1}{|\partial[i]|}, 
\end{align*}
showing that it is bounded from above by $O(\log \delta/\delta)$, where $\epsilon$ is a sufficiently small constant (in particular, $\epsilon \leq 1/8$).

We define a partition of the next $\epsilon d$ vertices $i \in \{j+1,\ldots,j+\epsilon \cdot \delta\}=:\mathcal{I}$ into two classes $1$ and $2$. Intuitively,
a vertex $i$ will be in class $1$ if most of its neighbours are vertices in $\{i+1,\ldots,n\}$, while a vertex $i$ will be in class $2$ if most of its neighbours are vertices in $\{1,\ldots,i\}$.

For technical reasons, however, it is advantageous to define the partition into classes based on the first $j$ vertices only. This way we ensure that the division into classes is invariant under any permutation of the vertices in $\mathcal{I}$, which turns out to be useful in the proof.


\textbf{Definition of Class 1.}
We say that a vertex $u \in \partial[j]$ is in Class $1$ if it has at least $(7/8) \delta$ neighbours in $V \setminus [j]$. Note that for any rank $i \in \mathcal{I}$ assigned to $u$ in the linear order, $u$ will contribute at least $(7/8) \delta - (i-j) \geq (7/8) \delta - \epsilon \delta = (3/4) \cdot \delta$ cut-edges towards $\partial[i]$, for any $\Delta_k$ with $k \in \mathcal{I}$.

Furthermore, any vertex $v \not\in \partial[j]$ will be also in Class $1$. Again, if $v$ is assigned any rank $i \in \mathcal{I}$, it follows that $v$ will contribute at least $\delta - (i-j) \geq \delta - \epsilon \delta \geq (3/4) \delta$ cut-edges towards $\partial[i]$, for any $\Delta_k$ with $i \leq k \leq j + \epsilon \delta$.

To summarise, we can conclude that whenever a vertex $u$ from Class $1$ is assigned any rank $i$ between $j+1$ and $j+\epsilon \cdot \delta$, then it will contribute at least $(3/4) \cdot \delta$ cut-edges towards $\Delta_k$, for any $i \leq k \leq j + \epsilon \cdot \delta$.

\textbf{Definition of Class 2.}
We say that a vertex $u \in \partial[j]$ is in Class $2$ if it is not in Class $1$. This implies that vertex $u$ must have at least $(1/8) \delta$ neighbours in $[j]$. Notice that if $i \in \mathcal{I}$ is the rank of vertex $u$, then for any $j \leq k \leq i-1$,
we have
\[
  |E(\{u\}, [i-1])| \geq |E(\{u\}, [j])| \geq (1/8) \delta,
\]
so $u$ contributes at least $(1/8) \delta$ cut-edges towards $\Delta_k$.

\textbf{Lower Bounding the Number of Cut-Edges.} Based on the partition of vertices in $\mathcal{I}$ into Class $1$ and Class $2$, we will now derive a lower bound on
$\partial[i]$ for any $i \in \mathcal{I}$. This lower bound will be universal in the sense that it will hold for any ordering of the remaining vertices $\{j+1,\ldots,n\}$. Intuitively, the worst-case scenario is if the first vertices in $\mathcal{I}$ are all in Class $2$, which may lead to a steady decrease in the edge-boundary. However, thanks to the definition of Class $1$, we eventually ``run out'' of vertices in Class $2$, which means that all remaining vertices will be in Class $1$, which will lead to a steady increase in the edge-boundary. Obviously, there may be orderings in which vertices in Class $1$ and Class $2$ are intermixed in an arbitrary manner, and the proof below has to cope with this case.

Let us proceed with the formal argument. We will denote by $\gamma$ the number of vertices in $\mathcal{I}$ which are in Class $2$. Note that $0 \leq \gamma \leq \partial(j)$. Obviously, among the vertices $\{j+1,\ldots,j+i\}$, at most $\min\{i,\gamma\}$ vertices can be in Class $2$, meaning that among the vertices $\{j+i+1,\ldots,j+\epsilon \delta$ we have at least $\gamma - \min\{i,\gamma\} = \max\{\gamma - i,0\}$ vertices in Class $2$. This implies for any $i$ with $j \leq i \leq j + \gamma - 1$,
\[
  |\partial[i]| \geq 
  |E(\{j+i+1,\ldots,j+\epsilon \delta \},[j]\})| \geq (\gamma - i) \cdot (1/8) \delta.
\]
For the boundary case $i=j+\gamma$, we use the lower bound based on the edge-connectivity to conclude that
\[
  |\partial[i]| \geq \rho.
\]
Furthermore, if $\gamma < \epsilon \cdot d$, then for any $j+\gamma < i \leq j + \epsilon \cdot \delta$, there must be at least $(i-\gamma)$ vertices among $\{j,j+1,\ldots,j+i\}$ which are in Class 1, which in turn implies
\[
 |\partial[i]| \geq (i-\gamma) \cdot (3/4) \cdot \delta \geq (i-\gamma) \cdot (1/8) \delta.
\]
Combining the three inequalities from above yields 
\[
 \Delta_j \leq 2 \sum_{k=1}^{\epsilon \delta} \frac{1}{k \cdot (1/8) \delta} + \frac{1}{\rho} \leq 2 \sum_{k=2}^{\epsilon \delta} \frac{1}{k \cdot (1/8) \delta} + \frac{1}{\rho} = O(\log \delta/\delta + 1/\rho).
\]
Since this holds for any $1 \leq j \leq n - \epsilon \delta - 1$, we conclude that by the second statement of Lemma~\ref{lem:edgeconn},
\begin{align*}
  2 \sum_{i=1}^{n/2} \frac{1}{|\partial[i]|} &\leq 2 \sum_{\lambda=0}^{n/(\epsilon \delta)-1} \Delta_{1+\lambda \cdot \epsilon \delta} \\ &\leq \left\lceil \frac{n}{\epsilon \delta} \right\rceil \cdot O(\log \delta/\delta + 1/\rho) \\ &= O(n/\delta^2 \cdot \log \delta + n/(\rho \delta)),
\end{align*}
as required. 
\end{proof}

\begin{lemma}
\label{lem:connected}
Let $g : V \to \R$ be a maximiser for \eq{var_hit}. Then, there exists a linear ordering of $V$ such that $0 = g(1) \le \cdots \le g(n) = 1$ and, for any $j=1,\dots,n$, $\{1,\dots,j\}$ forms a connected vertex-induced subgraph of $G$.
\end{lemma}
\begin{proof}
Assume w.l.o.g. that $s=1$ and $t=n$ and suppose by contradiction no such ordering exists. Then there exists $1 < j < n-1$ such that $\{1,\dots,j-1\}$ is connected, but $\{1,\dots,j\}$ is not, i.e. $j$ is adjacent only to vertices in $\{j+1,\dots,n\}$. Let $j'$ be the minimum index such that $j' > j$ and there exists an edge connecting $j'$ to $\{1,\dots,j-1\}$. Moreover, we can assume that $g(j') > g(j)$, otherwise we can reorder the vertices so that $\{1,\dots,j\}$ is connected. We construct a new function $h : V \to \R$ in this way: $h(i) = g(i)$ for any $i=1,\dots,j-1$, $h(i) = g(j')$ for any $i=j,\dots,j'$, and $h(i) = g(i)$ for any $i > j'$. Since there are no edges between $j$ and $\{1,\dots,j'-1\}$, but the graph is connected, there must be an edge $\{j,k\}$ with $k \in \{j',\dots,n\}$. But then we have that $(h(k)-h(j))^2 < (g(k)-g(j))^2$ and, in general, $h^{\intercal} \calL h < g^{\intercal} \calL g$. This implies that $g$ is not a maximiser for \eq{var_hit}, reaching a contradiction. 
\end{proof}

\begin{lemma}
\label{lem:optimalconn}
For any pair of $\rho$ and $d$ there is a graph matching the upper bound in Theorem~\ref{thm:connectivity} up to a factor of $O(\log d)$.
\end{lemma}
\begin{proof}
We show that, for any degree $d$ and edge-connectivity $\rho$, there exists a graph that matches the bound of \lemref{edgeconn} up to a $O(\log{d})$-factor. In fact, we take the same graph as in \cite[Example~6.24]{nonrevFill} by Aldous and Fill. In this example, we have $n$ vertices $\{0,1,\ldots,n-1\}$ and the edge set is $(i,i+u \mod n)$ for any $0 \leq i \leq n-1$ and $1 \leq u \leq \rho$. As they argue, the maximum commute time  is $\Theta(n^2 / \rho)$, while their bound only gives $\Theta(n^2 / \rho^{1/2})$. Note that our bound gives $\Theta( n^2 \cdot \log \rho/ \rho)$, which is tight up to a logarithmic factor.
\end{proof}
